\documentclass[11pt]{article}
\usepackage[utf8]{inputenc}
\usepackage{almost_sure}

\allowdisplaybreaks[4]

\newcommand{\Addresses}{{
  \bigskip\bigskip \small

  Bjoern Bringmann, \textsc{University of California, Los Angeles, Department of Mathematics, 520 Portola Plaza, Los Angeles, CA 90095}\\\nopagebreak
  Email address: \texttt{bringmann@math.ucla.edu}

}}
\title{Almost sure local well-posedness for a derivative nonlinear wave equation}
\author{Bjoern Bringmann}

\begin{document}
\pagenumbering{arabic}
\maketitle

\let\thefootnote\relax
\let\thefootnote\relax
\begin{abstract} \noindent
We study the derivative nonlinear wave equation \( - \partial_{tt} u + \Delta u = |\nabla u|^2 \) on \( \mathbb{R}^{1+3} \). The deterministic theory is determined by the Lorentz-critical regularity \( s_L = 2 \), and both local well-posedness above \( s_L \) as well as ill-posedness below \( s_L \) are known. In this paper, we show the local existence of solutions for randomized initial data at the super-critical regularities \( s\geq 1.984\). In comparison to the previous literature in random dispersive equations, the main difficulty is the absence of a (probabilistic) nonlinear smoothing effect. To overcome this, we introduce an adaptive and iterative decomposition of approximate solutions into rough and smooth components. In addition, our argument relies on refined Strichartz estimates, a paraproduct decomposition, and the truncation method of de Bouard and Debussche. 
\end{abstract}

\section{Introduction}
We consider the Cauchy problem for the nonlinear wave equation 
\begin{equation}\label{intro:eq_nlw}
\begin{cases}
-\partial_{tt} u + \Delta u = |\nabla u|^2,  \qquad\qquad \text{for} ~ (t,x)\in \mathbb{R}^{1+d} \\
u|_{t=0}=f_0 , ~ \partial_t u|_{t=0}=f_1 
\end{cases}~,
\end{equation}
with initial data \( (f_0,f_1)\in H_x^{s}(\mathbb{R}^d)\times H_x^{s-1}(\mathbb{R}^d)\) and dimension \( d \geq 2 \). The choice of the nonlinearity \( |\nabla u|^2 \) is mainly for simplicity, and the methods of this paper also apply to a general quadratic derivative nonlinearity. In particular, using the sign change \( u \mapsto -u \), one can convert \( |\nabla u|^2 \) into \( -|\nabla u|^2 \).\\
 The deterministic theory of \eqref{intro:eq_nlw} is by now well-understood. 
Due to the scaling symmetry of the equation, one expects local well-posedness in \( H_x^{s}(\mathbb{R}^d)\times H_x^{s-1}(\mathbb{R}^d)\) only for \( s \geq d/2\). Using Lorentz-transformations (cf. \cite{SS98,Tao06}) one obtains a second obstruction to local well-posedness, and the Lorentz-critical regularity is given by \( (d+5)/4 \). The local well-posedness of \eqref{intro:eq_nlw} in Sobolev spaces for 
\begin{equation*}
s> s_d:=\max \left( \frac{d}{2}, \frac{d+5}{4} \right)
\end{equation*}
was proven by Ponce-Sideris \cite{PS93}, Zhou \cite{Zhou03}, and Tataru \cite{Tataru99}. In contrast, the ill-posedness for \( s \leq s_d \) was proven by Lindblad \cite{Lindblad93, Lindblad96} for certain derivative nonlinear wave equations. In particular, a minor modification of the example on \cite[p. 511]{Lindblad93} applies to \eqref{intro:eq_nlw} in dimension \( d=3 \). We remark that the gap between the scaling-critical regularity and the well-posedness theory can be closed in Fourier-Lebesque spaces, see \cite{GN14,GT13,Grunrock11}. \\
The purpose of this paper is to understand whether the ill-posedness in low regularity spaces is witnessed by generic or only exceptional sets of initial data. This leads us to consider the Cauchy problem \eqref{intro:eq_nlw} for random initial data \( (f_0^\omega, f_1^\omega) \in H^s(\mathbb{R}^d)\times H^{s-1}(\mathbb{R}^d) \). After pioneering work by Bourgain \cite{Bourgain94,Bourgain96} on nonlinear Schrödinger equations and more recent work by Burq-Tzvetkov \cite{BT08I,BT08II} on nonlinear wave equations,  the study of dispersive PDE with random initial data has seen an enormous growth of interest. We mention only some additional references in the context of nonlinear wave equations \cite{BB14,Bringmann18,BT14,S11,DLM17,DLM18,LM13,LM16,OP16,Pocovnicu17}. We also refer the reader to the survey paper \cite{BOP18} for a summary and further relevant references.
In this paper, we construct the random initial data using the Wiener randomization \cite{BOP15,LM13}. For this, let \( f \in L^2(\mathbb{R}^d) \) be arbitrary but fixed. Let \( \varphi\colon \mathbb{R}^d \rightarrow [0,1] \) be a smooth and compactly supported function such that the translates \( \{ \varphi(\cdot -k): k \in \mathbb{Z}^d \} \) form a partition of unity. Then, the Wiener decomposition of \( f \) is given in frequency space by
\begin{equation}\label{intro:eq_decomp}
\widehat{f}(\xi) = \sum_{k\in \mathbb{Z}^d} \varphi(\xi-k) \widehat{f}(\xi)~. 
\end{equation}
The Wiener randomization is now defined by randomizing the coefficients in \eqref{intro:eq_decomp}.  Let \( (\Omega, \mathscr{F}, \mathbb{P} ) \) be a probability space and let 
\( \{ g_k(\omega) \colon k \in \mathbb{Z}^d \} \) be a family of independent standard complex Gaussians. Then, we define
\begin{equation}\label{intro:eq_randomization}
\widehat{f^\omega}(\xi) = \sum_{k\in \mathbb{Z}^d} g_k(\omega) \varphi(\xi-k) \widehat{f}(\xi)~. 
\end{equation}
Thus, \( f^\omega \) is a random linear combination of functions that are frequency localized on cubes of scale \( \sim 1 \). The Gaussians may also be replaced by any family of independent uniformly sub-Gaussian random variables. Furthermore, if \( \varphi(\xi) = \varphi(-\xi) \) and \( f \) is real-valued, one can condition on the event that \( g_k = \overline{g_{-k}} \) for all \( k \in \mathbb{Z}^d \) to obtain real-valued functions \( f^\omega \).  \\

The first probabilistic result on wave equations with a quadratic derivative nonlinearity was recently obtained in \cite{CCMNS18}. The authors proved the following theorem. 

\begin{thm}[{\cite[Theorem 2.1 and Remark 2.3]{CCMNS18}}]\label{thm:staffilani}
Let \( (f_0 , f_1) \in H^1(\mathbb{R}^2)\times L^2(\mathbb{R}^2) \) and let \( (f_0^\omega,f_1^\omega) \) be as in \eqref{intro:eq_randomization}. Let \( F^\omega(t) = W(t)(f_0^\omega,f_1^\omega) \) be the solution to the linear wave equation with initial data \( (f_0^\omega,f_1^\omega) \). Furthermore, let \( u^{(j)} \) be the \( j\)-th Picard iterate, which is given by
\begin{align*}
u^{(0)}(t) &:= F^\omega(t)~, \\
u^{(j)} (t) &:= F^\omega(t) + \duh |\nabla u^{(j-1)}|^2 \dtp \qquad \forall j\geq 1~. 
\end{align*}
For any sufficiently small \( T>0 \), we have for almost every \( \omega \in \Omega \) that
\begin{equation*}
(u^{(j)}, \partial_t u^{(j)} ) \in \big( C_t^0 \dot{H}^1_x \times C_t^0 L^2_x\big)([0,T]\times \mathbb{R}^2) \qquad \forall j\geq 1~. 
\end{equation*}
\end{thm}
\begin{rem}
In fact, the theorem in \cite{CCMNS18} is slightly more general, and holds for any dimension \( d=2,3,4 \) and any quadratic derivative nonlinearity. Furthermore, the randomization in \cite{CCMNS18} uses random signs instead of Gaussians.
\end{rem}
The randomness in Theorem \ref{thm:staffilani} is essential. For deterministic data, the statement of the theorem may even fail for the first iterate \( u^{(1)} \), see \cite{FK00,Zhou97}. The bounds in \cite{CCMNS18} on the size of \( u^{(j)} \), however,  are not uniform in \( j\geq 0\), and are not sufficient to conclude the existence of a solution.  In fact, proving the existence of solutions for random initial data is mentioned as an open problem on \cite[p.3]{CCMNS18}. \\
 The main theorem of this paper solves this problem (in three dimensions) for certain Lorentz super-critical regularities \( s < 2 = s_3  \).

\begin{thm}[Main theorem]\label{main_thm}
Assume that \( (f_0,f_1) \in H_x^s(\rthree)\times H_x^{s-1}(\rthree)\), where \( s \geq 1.984 \). In addition, let \( 0 < T_0 \ll 1 \) and \( \sigma=1.1 \). Then, there exists a random function \( u \) and random times \( 0 < T(\omega) \leq T_0 \) such that 
\begin{equation}\label{intro:eq_u_space}
\begin{aligned}
u&\in  \big( L_\omega^2 C_t^0 H_x^s \medcap L_\omega^2 L_t^2 W_{x}^{\sigma,\infty} \big)(\Omega \times [0,T_0] \times \rthree)~,\\
\partial_t u &\in\big( L_\omega^2 C_t^0 H_x^{s-1} \big)(\Omega \times [0,T_0] \times \rthree)~,
\end{aligned}
\end{equation}
and such that for almost every \( \omega \in \Omega \) it holds that 
\begin{equation}\label{intro:eq_u_duhamel}
u(t)= W(t) (f_0^\omega, f_1^\omega) + \duh |\nabla u(t^\prime)|^2 \dtp \qquad \forall t \in [0,T(\omega)]~. 
\end{equation}
\end{thm}
\begin{rem}
A minor modification of the arguments should lead to a similar result in dimension \( d=2 \). We expect the restriction \( s \geq 1.7281 \) in \( H^s_x(\mathbb{R}^2) \times H^{s-1}_x(\mathbb{R}^2) \), which lies below the Lorentz critical regularity \( s=1.75 \). In contrast, the extension to high-dimensions \( d \geq 4 \) may be more difficult, and likely involves \( X^{s,b} \)-type spaces \cite{Bourgain93}. The techniques of this paper may also apply to nonlinear wave equations with null-forms \cite{KM93}, but we have not pursued this direction yet. 
\end{rem}
In the following we sketch the main ideas behind the proof of Theorem \ref{main_thm}. We first describe why a common combination of Bourgain's trick \cite{Bourgain96}, which is related to the Da Prato-Debussche trick \cite{PD02}, and nonlinear smoothing estimates cannot be applied to \eqref{intro:eq_nlw}. As above, let \( F^\omega(t) \) be the solution to the linear wave equation with initial data \( (f^\omega_0,f^\omega_1) \). Then, we decompose the solution as \( u(t)= F^\omega(t)+w(t) \), and obtain the equation
\begin{equation}\label{intro:eq_nlw_nl}
\begin{cases}
-\partial_{tt} w + \Delta w =  |\nabla w|^2 + 2 \nabla w \cdot \nabla F^\omega + |\nabla F^\omega|^2 \qquad\qquad \text{for} ~ (t,x)\in \mathbb{R}^{1+d} \\
w|_{t=0}=0, ~ \partial_t w|_{t=0}=0
\end{cases}~.
\end{equation}
Following Bourgain's work \cite{Bourgain96}, one can try to construct a solution \( w(t) \) of  \eqref{intro:eq_nlw_nl} through a contraction mapping argument at a sub-critical regularity \( \nu> s_d \). In addition to probabilistic Strichartz estimates for \( F^\omega(t) \), this requires a (probabilistic) nonlinear smoothing estimates for \( w(t) \). For example, in the case of the nonlinear Schrödinger equation, this can be proven using either bilinear Strichartz estimates \cite{BOP15,BOP17,Bourgain96,Brereton16} or local smoothing estimates \cite{DLM18}. However, the equation \eqref{intro:eq_nlw_nl} does not exhibit nonlinear smoothing. To see this, we examine the low-high interaction term \( \nabla P_1 F(t) \cdot \nabla P_{\gg 1} F(t) \). Heuristically, we have for any \( \nu >s_d>s \) that 
\begin{equation*}
|\nabla|^{\nu} \duh \nabla P_1 F^\omega \cdot \nabla P_{\gg 1} F^\omega \dtp \simeq \int_0^t \sin((t-t^\prime)|\nabla|)~ \nabla P_1 F^\omega \cdot |\nabla|^{\nu-1} \nabla P_{\gg 1} F^\omega \dtp
\end{equation*}
Thus, the linear evolution \( F^\omega(t) \) is attacked by more than \( s \) derivatives. Since the Duhamel integral does not increase the spatial regularity, and the bilinear Strichartz estimates for the wave equation do not gain spatial derivatives, we cannot show a nonlinear smoothing estimate for this term. In fact, by choosing the initial data to be frequency localized on two cubes of scale \( \sim 1\), one at distance \( \sim 1 \) and one at distance \( \sim N \gg 1 \) from the origin, we see that this term may have the same spatial regularity as the initial data. We remark, however, that there are bilinear estimates which gain derivatives in null directions, see e.g. \cite{DFS12,FK00,KRT02}, and the references therein.\\
In the above heuristic, we have seen that the low-high interactions form the main obstacle towards the well-posedness of \eqref{intro:eq_nlw_nl} at a regularity \( \nu> s_d \). In other dispersive equations, such as the Benjamin-Ono equation, the low-high interactions can be removed by using a gauge transformation \cite{Tao04}. We refer to \cite{Deng15} for the (difficult) implementation of this idea in a probabilistic setting. Unfortunately, \eqref{intro:eq_nlw_nl} does not appear to have such a gauge transformation. Instead, we remove the low-high interactions by viewing them as part of the linear evolution for the high-frequency data. To make this precise, we first need to introduce an iterative method. For \( n\geq 0 \) and \( N= 2^n \), we set 

\begin{equation*}
 Q_1 f^\omega(x):= g_0(\omega) P_0 f(x)\quad \text{and} \quad Q_{N} f^\omega(x) := \sum_{N/2\leq \| k \|_2 < N} g_k(\omega) P_k f (x), ~\text{where}~  N \geq 2~.
 \end{equation*}
 We remark that the family of random functions \( \{ Q_N f \}_{N\geq 1} \) is jointly independent, which is essential for the argument. Furthermore, we define
 \begin{equation*}
 Q_{\leq N}f^\omega(x):= \sum_{M\leq N} Q_{M}f^\omega(x)~.
 \end{equation*}
 Since the frequency-truncated initial data is smooth, there exists a solution \( u_n \) of 

 \begin{equation}\label{intro:eq_un}
\begin{cases}
 -\partial_{tt} u_n + \Delta u_n = |\nabla u_n|^2\\
 u_n|_{t=0} = Q_{\leq N} f_0^\omega~,~~\partial_t u_n |_{t=0} = Q_{\leq N} f_1^\omega ~.
\end{cases}
\end{equation}
Our goal is to prove the convergence of \( u_n \) in the low regularity space \( C_t^0 H_x^s \), and define the solution \( u \) as the limit of the sequence \( u_n \). First, we define the increment \( v_n \) by writing \( u_n = u_{n-1} + v_n \). To simplify the notation, we use the convention \( u_{-1} = 0\). Then, the equation for \( v_n \) reads 

\begin{equation}\label{intro:eq_vn}
\begin{cases}
 -\partial_{tt} v_n + \Delta v_n = |\nabla v_n|^2 + 2 \nabla u_{n-1} \cdot \nabla v_n   \\
 v_n|_{t=0} = Q_{N} f_0^\omega~,~ ~ \partial_t v_n |_{t=0} = Q_{N} f_1^\omega ~.
\end{cases}
\end{equation}
To control \( v_n \) uniformly in \( n \geq 0 \), it is necessary to decompose it into a rough, linear component and a smooth, nonlinear component. For a fixed parameter \( \gamma \in (0,1) \), we define the adapted linear evolution \( F_n^\omega \) as the solution to 

\begin{equation}\label{intro:eq_Fn}
\begin{cases}
-\partial_{tt} F_{n}^\omega+\Delta F_{n}^\omega =  2  \nabla P_{\leq N^\gamma} u_{n-1} \cdot \nabla F_{n}^\omega\\
F_{n}^\omega|_{t=0} = Q_N f_0^\omega~,~~ \partial_t F_{n}^\omega|_{t=0} = Q_N f_1^\omega~. 
\end{cases}
\end{equation}
As a consequence, the equation for the nonlinear component \( w_n = v_n -F_n^\omega \) is given by 
 \begin{equation}\label{intro:eq_wn}
 \begin{cases}
 -\partial_{tt} w_n + \Delta w_n =|\nabla F_n^\omega + \nabla w_n|^2 + 2 \nabla u_{n-1} \nabla w_n + 2  \nabla P_{>N^\gamma} u_{n-1}  \cdot \nabla F_n^\omega~,\\
 w_n|_{t=0} = 0~, ~~ \partial_t w_n|_{t=0} =0~. 
 \end{cases}
 \end{equation}
 To obtain the lowest regularity \( s \), we will later choose \( \gamma=0.88 \), see \eqref{fin:eq_parameters}. 
 Therefore, the inhomogeneous term \( \nabla P_{> N^\gamma} u_{n-1} \cdot \nabla F_n^\omega \) in \eqref{intro:eq_wn} is essentially a high-high interaction. We can then hope to control \( w_n \) at a higher regularity than \( F_n^\omega \). \\
After this description of our iteration scheme and decomposition, we now mention the remaining difficulties in the implementation. Even though \eqref{intro:eq_Fn} is linear in \( F_n^\omega \), it is highly nonlinear in the random variables \( \{ g_k \colon \| k \|_2 < N/2 \} \). The resulting difficulties on the probabilistic side of the argument can be solved using the truncation method of de Bouard and Debussche \cite{BD99}. In order to prove probabilistic Strichartz estimates for \( F_n^\omega \), one needs to control the effect of the variable-coefficient term \( \nabla P_{\leq N^\gamma} u_{n-1} \nabla F_n^\omega \) on the frequency support of \( F_n^\omega \). For this, we rely on refined Strichartz estimates and re-centered Besov-type spaces. Finally, we control the nonlinear component \( w_n \). To handle the low-high interaction term \( \nabla P_1 w \cdot \nabla F_n^\omega \), we place \( w_n \) in the space of frequency-localized functions \( \YN \), which is defined in \eqref{prelim:eq_YN}.\\

\begin{rem} In the end of this introduction, we now mention a related method of Bourgain. In \cite{Bourgain97}, Bourgain proves the invariance of the Gibbs measure for a certain Gross-Pitaevski equation. To this end, he examines the Cauchy problem 
 \begin{equation}\label{intro:eq_gross_pitaevski}
\begin{cases}
i \partial_t u + \Delta u + (V * |u|^2) u = 0,   \qquad (t,x)\in \mathbb{R}\times \mathbb{T}^3\\
 u|_{t=0} = \phi^\omega.
\end{cases}
\end{equation}
The interaction potential \( V \) satisfies \( |\widehat{V}(k)|\lesssim \langle k \rangle^{-\beta} \), where \( \beta >2 \), and \( \widehat{V}(0) = 0 \) (after a renormalization). The random data is given by \( \phi^\omega = \sum_{k\in \mathbb{Z}^3} g_k(\omega) \langle k \rangle^{-1} e^{ikx} \), and hence corresponds to a typical sample of the Gibbs measure.  The method of \cite{Bourgain97} combines a quasi-linear iteration scheme (cf. \cite[(3.8)]{Bourgain97}) with a detailed analysis of a power series expansion (cf. \cite[(3.41)]{Bourgain97}). It also has some similarities with the method of this paper, see e.g. Proposition \ref{lin:prop_freq_envelope} and \cite[(5.6)]{Bourgain97}.  In contrast to the derivative nonlinear wave equation \eqref{intro:eq_nlw}, however,  the Gross-Pitaevski equation \eqref{intro:eq_gross_pitaevski} exhibits a nonlinear smoothing effect. In fact, even though \( \phi^\omega \) only has Sobolev-regularity \( s= -1/2 - \), one can show that 
\begin{equation*}
\int_0^t e^{it^\prime\Delta} (V*|e^{it^\prime\Delta} \phi^\omega|^2) e^{it^\prime\Delta} \phi^\omega \, \mathrm{d}t^\prime \in H_x^{0-}(\mathbb{T}^3)~\text{a.s.} 
\end{equation*}
The high-high interactions will experience smoothing through the potential \( V \), while the low-high interactions experience smoothing through a bilinear dispersive effect. \\
Due to the nonlinear smoothing effect in \eqref{intro:eq_gross_pitaevski}, it is unclear to the author whether Bourgain's method can be extended to \eqref{intro:eq_nlw}, and we leave this question for future research. Conversely, it would be interesting to know if modern methods can improve the condition \( \beta > 2 \) on the interaction potential \( V \). 
\end{rem}

\paragraph{Acknowledgements}~\\
The author thanks his doctoral advisor Terence Tao for his invaluable guidance and support. The author also thanks Rowan Killip and Monica Visan for many interesting discussions. Furthermore, the author thanks the anonymous referees for several helpful comments.

\section{Notation and Preliminaries}
In this section, we will provide the necessary notation and preliminaries for the rest of the paper. In Section \ref{sec:function_spaces}, we construct spaces of frequency-localized functions. In Section \ref{sec:strichartz}, we recall the Strichartz estimates for the wave equation. In particular, we describe the refinement of Klainerman and Tataru \cite{KT99}. 

\subsection{Function Spaces}\label{sec:function_spaces}
For any function \( f\in L^1(\mathbb{R}^d) \), we define its Fourier transform \( \widehat{f} \)  by
\begin{equation*}
\widehat{f}(\xi) := \frac{1}{(2\pi)^{\frac{d}{2}}} \int_{\mathbb{R}^d} \exp(-i x \cdot \xi) f(x) \dx~. 
\end{equation*}
Let \( \varphi \colon \mathbb{R}^d \rightarrow \mathbb{R} \) be a smooth, compactly supported function s.t. \( \varphi|_{B(0,1)} \equiv 1 \) and \( \varphi|_{\mathbb{R}^d\backslash B(0,2)} \equiv 0 \). We set \( \psi_1(\xi) = \varphi(\xi) \) and \( \psi_M(\xi) := \varphi(\xi/M)- \varphi(2\xi/M) \), \( M \geq 2 \). For any dyadic \( M \geq 1 \), we define the re-centered Littlewood-Paley operators by
\begin{equation*}
\widehat{P_{M;k}f}(\xi):= \psi_M(\xi-k) \widehat{f}(\xi)  
\end{equation*}
The standard Littlewood-Paley projections \( P_M \) are given by \( P_{M;0} \).  We also use the fattened Littlewood-Paley projections \( \widetilde{P}_M \), which are defined using multipliers \( \widetilde{\psi}_M \) with slightly larger support. \\
The following function spaces are partly motivated by the frequency envelopes in \cite{Tao01,Tao04}. 
We first define two weight functions \( c \colon 2^\mathbb{N} \rightarrow \mathbb{R}_+ \). Let \( N\geq 1 \) be a fixed dyadic integer and  let \( D > 0 \) be arbitrary.\\
To capture functions that are localized at frequencies \( \sim N \), we set
\begin{equation}\label{prelim:eq_cN}
c_{N,D}(M):= \max \left( \frac{N}{M}, \frac{M}{N} \right)^{D}~.
\end{equation}
In addition, to capture functions localized at frequencies \( \lesssim N \), we set 
\begin{equation}\label{prelim:eq_clN}
c_{\leq N,D}(M):= \max \left( 1,\frac{M}{N} \right)^{D}~.
\end{equation}
Next, let \( u\colon \mathbb{R}^{1+3} \rightarrow \mathbb{R} \) be a function on space-time. We define frequency localized versions of the \( L_t^\infty L_x^2 \)-norm by
\begin{align}
\| u \|_{\XNT}&:= \sum_{M\geq 1} c_{N,D}(M) \| P_M u \|_{L_t^\infty L_x^2([0,T]\times \rthree)}~,\\
\| u \|_{\XlNT}&:= \sum_{M\geq 1} c_{\leq N,D}(M) \| P_M u \|_{L_t^\infty L_x^2([0,T]\times \rthree)}~.
\end{align}
Similarly, we define frequency localized versions of the Strichartz-type \( L_t^2 L_x^\infty \)-norm by 
\begin{align}
\| u \|_{\SNT}&:= \sum_{M\geq 1} c_{N,D}(M) \| P_M u \|_{L_t^2 L_x^\infty([0,T]\times \rthree)}~,\\
\| u \|_{\SlNT}&:= \sum_{M\geq 1} c_{\leq N,D}(M) \| P_M u \|_{L_t^2 L_x^\infty([0,T]\times \rthree)}~.
\end{align}
The function spaces corresponding to the norms above are given by 
\begin{equation}\label{prelim:eq_spaces}
\begin{aligned}
\XNT &:= \{ u \in C_t^0 L_x^2([0,T]\times \rthree)\colon \| u \|_{\XNT} < \infty \}~, \\
\XlNT& := \{ u \in C_t^0 L_x^2 ([0,T]\times \rthree)\colon \| u \|_{\XlNT} < \infty \}~, \\
\SNT &:= \{ u \in L_t^2 L_x^\infty([0,T]\times \rthree)\colon \| u \|_{\SNT} < \infty \}~, \\
\SlNT &:= \{ u \in L_t^2 L_x^\infty([0,T]\times \rthree ) \colon \| u \|_{\SlNT} < \infty \}~. 
\end{aligned}
\end{equation} 
Note that the \( \XNT \) and \( \XlNT \)-spaces only contain functions in \( C_t^0 L_x^2([0,T]\times \rthree) \). We now record some basic properties of these spaces. 

\begin{lem}\label{prelim:lem_basic_properties}
Let \( N\geq 1 \) be a fixed dyadic integer and  let \( D > 0 \) be arbitrary. Then, the spaces \( \XNT \), \( \XlNT \), \( \SNT \), and \( \SlNT \) equipped with their corresponding norms are complete. \\
Furthermore, for each \( u \in \XNT \), the mapping
\begin{equation}\label{prelim:eq_continuity}
t \in [0,T] \mapsto \| u \|_{\XN([0,t])} 
\end{equation}
is continuous. An analogous continuity statement also holds for the other function spaces. 
\end{lem}
The continuity of \eqref{prelim:eq_continuity} is important in the proof of Proposition \ref{nl:prop_wn}, which uses a contraction mapping argument.
\begin{proof}
The completeness follows from standard arguments in real analysis, and the proof is omitted. \\
It remains to show the continuity statement \eqref{prelim:eq_continuity}. Since \( w\in C_t^0 L_x^2([0,T]\times \rthree)\), each individual summand \( t \mapsto \| P_M u \|_{L_t^\infty L_x^2([0,T]\times \rthree)} \) is continuous. Since \( \| u \|_{\XN([0,t])} \) is a uniform limit of the partial sums in \( M \geq 1 \), the result follows. 
\end{proof}
Equipped with the functions spaces above, we are now ready to define the function space \( \YN \) for the solution \( w_n \) of \eqref{intro:eq_wn}. For given parameters \( \nu > 2 \), \( \sigma = \nu-1 - \), and \( \eta,D > 0 \), we set 
\begin{equation}\label{prelim:eq_YN}
\begin{aligned}
\YN([0,T]):=\{& u \colon [0,T]\times \rthree \rightarrow \mathbb{R} | ~ \bra^\nu u, \bra^{\nu-1} \partial_t u\in ( \XNeta \medcap \XlN)([0,T]), \\
& \text{and} \quad \bra^\sigma u \in (\SNeta \medcap \SlN)([0,T]) \}~. \\
\end{aligned}
\end{equation}
The corresponding norm is defined by
\begin{equation*}
\begin{aligned}
\| u \|_{\YN([0,T])} :&= \| \bra^\nu u \|_{(\XNeta \bigcap \XlN)([0,T])} +  \| \bra^{\nu-1} \partial_t u \|_{(\XNeta \bigcap \XlN)([0,T])}  \\
&+ \| \bra^\sigma u \|_{(\SNeta \bigcap \SlN)([0,T])}~. 
\end{aligned}
\end{equation*}
The main regularity parameter is \( \nu > 2 \), and it describes the number of derivatives of \( w_n \) that are controlled in the \( L_t^\infty L_x^2 \)-type norm. The value of \( \sigma \) is then determined by the deterministic Strichartz estimates. Finally, the parameters \( \eta> 0 \) and \( D >0 \) describe the localization to frequencies \( \sim N \) and \( \lesssim N \), respectively. Due to high-high to low frequency interactions in the quadratic term \( |\nabla w_n|^2 \), we have to choose \( \eta < \nu - 1 \). In contrast, there is essentially no transfer from low to high frequencies over short time intervals, and hence \( D >0 \) can be chosen arbitrarily large. \\
The (nearly) optimal choice of the parameters leads to \( \nu = 2.1001 \), see \eqref{fin:eq_parameters}. This may seem surprising, since this is an absolute amount above the Lorentz critical regularity \( s_3 = 2 \). The additional regularity is used to control the effect of the variable-coefficient term \( \nabla P_{\leq N^\gamma} u_{n-1} \cdot \nabla F_n^\omega \) on the frequency support of the randomized initial data, see Proposition  \ref{lin:prop_freq_envelope}.

Recall that the atoms in the Wiener randomizaton are localized in frequency space to cubes of scale \( \sim 1 \). To take advantage of this, we introduce the following Besov-type spaces.
Let \( \gamma \in (0,1) \) and \( k \in \mathbb{Z}^3 \) with \( \| k \|_2 \sim N \). We define the weight function
\begin{equation}\label{prelim:eq_ck}
c^{\rho,\gamma}_{k,D}(M) := M^\rho \max \left( 1, \frac{M}{N^\gamma} \right)^{D}~. 
\end{equation}
Using this weight function, we set
\begin{equation*}
\| f \|_{\Bk} := \sum_{M\geq 1} c^{\rho,\gamma}_{k,D}(M) \| P_{M;k} f \|_{L_x^2(\rthree)} \quad \text{and} \quad \Bk := \{ f \in L_x^2(\rthree)\colon \| f \|_{\Bk}  < \infty \}~. 
\end{equation*}

\subsection{Strichartz Estimates}\label{sec:strichartz}

First, we state a local Strichartz estimate in the form needed for this paper. 

\begin{lem}[Strichartz Estimate]\label{prelim:lem_strichartz}
Let \( \nu> 2 \) and let \( \sigma= \nu - 1 - \delta \), where \( \delta >0 \) is small. Let \( 0 < T \leq 1 \) and let \( u \) be a solution of 
\begin{equation}\label{prelim:eq_strichartz}
\begin{cases}
-\partial_{tt} u + \Delta u = F  \qquad\qquad \text{for} ~ (t,x)\in [0,T]\times \rthree \\
u|_{t=0}=f_0 , ~ \partial_t u|_{t=0}=f_1 
\end{cases}~.
\end{equation}
Then, we have that 
\begin{align*}
&\| u \|_{C_t^0 H_x^\nu([0,T]\times \rthree)}+\| \partial_t u \|_{C_t^0 H_x^{\nu-1}([0,T]\times \rthree)} + \| \langle \nabla \rangle^{\sigma} u \|_{L_t^2 L_x^\infty([0,T]\times \rthree)} \\
&\lesssim_{\nu,\sigma} \| f_0 \|_{H_x^\nu(\rthree)} + \| f_1 \|_{H_x^{\nu-1}(\rthree)} + \| \langle \nabla \rangle^{\nu-1} F \|_{L_t^1 L_x^2([0,T]\times \rthree)}~. 
\end{align*}
In particular, \( u \in C_t^0 H_x^\nu([0,T]\times \rthree) \) and \( \partial_t u \in C_t^0 H_x^{\nu-1}([0,T]\times \rthree) \). 
\end{lem}

\begin{proof}
This lemma follows directly from the (global) Strichartz estimates in \cite{KT98}. In order to deal with the inhomogeneous norms, we also use that 
\begin{equation*}
\Big \| \frac{\sin(t|\nabla|)}{|\nabla|} \Big \|_{H_x^{\nu-1}(\rthree) \rightarrow H_x^{\nu}(\rthree)} \leq \max( 1, |t|) = 1~. 
\end{equation*}
The local estimate with an \( \epsilon\)-loss at the endpoint \( (2,\infty) \) follows from Hölder's inequality, Bernstein's estimate, and \cite[Corollary 1.3]{KT98} with \( (q,p)=(2+,\infty-)\).  
\end{proof}

In the following, we recall a refined Strichartz estimate from \cite{KT99}. This estimate has already been used in the context of the Wiener randomization in \cite{DLM17}. 
\begin{lem}[{Refined Strichartz Estimate \cite{KT99}}]\label{prelim:lem_refined_strichartz}
Assume that \( k \in \mathbb{Z}^3 \) with \( \| k \|_2 \sim N \), and let \( 1 \leq M \ll N \). Furthermore, let \( (q,p) \) be a sharp wave-admissible Strichartz pair, i.e., \( 2 \leq q,p < \infty \) and 
\begin{equation*}
\frac{1}{q}+ \frac{1}{r} = \frac{1}{2}~. 
\end{equation*}
Then, it holds for all \( T>0 \) that
\begin{equation*}
\Big \| \duh P_{M;k} F \dtp \Big \|_{L_t^q L_x^p([0,T]\times \rthree)} \lesssim \Big( \frac{M}{N} \Big)^{\frac{1}{2}-\frac{1}{p}} N^{-1} N^{\frac{3}{2}-\frac{1}{q}- \frac{3}{p}} \| F \|_{L_t^1 L_x^2([0,T]\times \rthree)}~. 
\end{equation*}
\end{lem}
The refined Strichartz estimate exhibits a gain in \( M/N \). Since the projection onto small balls at a large distance from the origin essentially rules out the Knapp counterexamples, this is to be expected. 

\section{The truncated equations}\label{sec:it} 
Recall from the introduction that \( u_n \), \( F_n^\omega \), and \( w_n \) are supposed to solve \eqref{intro:eq_un}, \eqref{intro:eq_Fn}, and \eqref{intro:eq_wn}. However, we cannot directly work with the weak formulation of these equations. The problem is unrelated to any estimates in the deterministic part of the argument, and comes only from the moments with respect to \( \omega \in\Omega \). Let us describe the problem by examining \eqref{intro:eq_wn}, which determines the nonlinear component \( w_n \). Since there is no gain of integrability in \( \omega \), the quadratic term \( |\nabla w_n|^2 \) prevents us from using a contraction mapping argument in \( L_\omega^r L_t^q L_x^p \)-type spaces. Nevertheless, by arguing pointwise in \( \omega\), one could construct a solution \( w_n \) of \eqref{intro:eq_wn} on a random time interval \( [0,T_n(\omega)] \). Unfortunately, \( T_n(\omega) \) would also depend on the nonlinear solution \( u_{n-1} \). Since the nonlinear solution \( u_{n-1} \) depends in a complicated fashion on the random variables, it would then be difficult to control \( T_n(\omega)  \) pointwise in \( \omega \) as \( n \rightarrow \infty\). Using the truncation method of de Bouard and Debussche \cite{BD99}, we can circumvent this problem. The main idea is to truncate the nonlinearity of \eqref{intro:eq_wn}, and then work on a fixed deterministic time interval. Due to the truncation, we can use a contraction mapping argument in \( L_\omega^r L_t^q L_x^p\)-type spaces, and also obtain much simpler nonlinear estimates. After all iterates have been constructed, one can remove the truncation by restricting to a small random time interval. \\
Let us also briefly explain why the truncation method is absent from previous work on random dispersive equations. In previous methods, the construction of rough objects, such as \( F_n^\omega\), does not depend on the solution to a nonlinear equation. As a result, they only require a single contraction mapping argument, and it suffices to work on a single random time interval \( [0,T(\omega)] \). We also refer the reader to \cite[Remark 3.5 and 3.7]{BOP18}, which explain the underlying separation between probabilistic and analytic arguments.

After this motivation, we now describe the truncation method of de Bouard and Debussche \cite{BD99}.
 Let \( \theta \colon \mathbb{R}_{\geq 0} \rightarrow [0,1] \) be a smooth function s.t. \( \theta|_{[0,1]} = 1 \) and \( \theta|_{[2,\infty)} = 0 \). We want to define the truncated solutions \( u_{n,\theta}, F_{n,\theta}^\omega,\) and \( w_{n,\theta}\). To simplify the notation, we write \( M=2^m \) and \( N=2^n \). We first define the cutoff functions
 \begin{align}
 \theta_{F,w;\leq n-1}(s)
 &:= \theta\left(\sum_{m=0}^{n-1} \Big(\| \langle \nabla \rangle^{\sigma^\prime} F_{m,\theta}^\omega \|_{\SlMprime([0,s])} +\| \langle \nabla \rangle^{\sigma} w_{m,\theta} \|_{\SlM([0,s])}+ \| \langle \nabla \rangle^{\nu} w_{m,\theta} \|_{ \XlM([0,s])}\Big) \right)~, \notag \\
\theta_{F;n}(s) &:= \theta\left( \| \langle \nabla \rangle^{\sigma^\prime} F_{n,\theta}^\omega \|_{\SNprime([0,s])}\right)~, \notag\\
\theta_{w;n}(s) &:= \theta\left( \| \langle \nabla \rangle^{\sigma} w_{n,\theta} \|_{\SlN([0,s])}+ \| \langle \nabla \rangle^{\nu} w_{n,\theta} \|_{ \XlN([0,s])}\right)~. \label{it:eq_theta_wn}
\end{align}
Let \( F_{\leq n-1,\theta}^\omega := \sum_{m=0}^{n-1} F_{m,\theta}^\omega \) and \( w_{\leq n-1,\theta} := \sum_{m=0}^{n-1} w_{m,\theta} \). 
For future use, we remark that 
\begin{equation}\label{it:eq_triangle}
\begin{aligned}
\| \bra^{\sigma^\prime} F_{\leq n-1,\theta}^\omega \|_{\SlNprime}&= \| \sum_{m=0}^{n-1} \bra^{\sigma^\prime} F_{m,\theta}^{\omega} \|_{\SlNprime} \leq \sum_{m=0}^{n-1} \| \bra^{\sigma^\prime} F_{m,\theta}^\omega \|_{\SlNprime} \leq \sum_{m=0}^{n-1} \| \bra^{\sigma^\prime} F_{m,\theta}^\omega \|_{\SlMprime} ~,\\
\| \bra^\sigma w_{\leq n-1,\theta} \|_{\SlN} &= \| \sum_{m=0}^{n-1} \bra^{\sigma} w_{m,\theta} \|_{\SlN} ~ \leq \sum_{m=0}^{n-1} \| \bra^{\sigma} w_{m,\theta}\|_{\SlN} ~  \leq \sum_{m=0}^{n-1} \| \bra^{\sigma} w_{m,\theta}\|_{\SlM}~.
\end{aligned}
\end{equation}
Then, we let \( F_{n,\theta}^\omega \) be a solution of the truncated equation
\begin{equation}
F_{n,\theta}^\omega(t)= W(t) (Q_N f_0^\omega, Q_N f_1^\omega) + 2 \duh \theta_{F,w;\leq n-1}(s) P_{\leq N^\gamma} \nabla u_{n-1,\theta}(t^\prime) \nabla F_{n,\theta}^\omega(t^\prime) \dtp~. \label{it:eq_truncated_Fn}
\end{equation}
In Section \ref{sec:lin}, it will be useful to decompose \( F_{n,\theta}^\omega \) into a superposition of the solutions corresponding to each individual individual pair \( (\Pk f_0, \Pk f_1) \). Thus, we define \( F_{n,k,\theta} \) as the solution of 
\begin{equation}
F_{n,k,\theta}(t)= W(t) (\Pk  f_0, \Pk f_1) + 2 \duh \theta_{F,w;\leq n-1}(t^\prime) P_{\leq N^\gamma} \nabla u_{n-1,\theta}(t^\prime) \nabla F_{n,k,\theta}(t^\prime) \dtp~. \label{it:eq_truncated_Fnk}
\end{equation}
The nonlinear component \( w_{n,\theta}(t) \) is defined as the solution of 
\begin{align}
w_{n,\theta}(t) &= \duh  \theta_{F;n}(t^\prime) |\nabla F_{n,\theta}^\omega|^2\dtp
		+ 2 \duh \theta_{F;n}(t^\prime) \nabla F_{n,\theta}^\omega \nabla w_{n,\theta} \dtp   \notag\\
		&+ \duh \theta_{w;n}(t^\prime) |\nabla w_{n,\theta}|^2\dtp\label{it:eq_truncated_wn}  
		+ 2 \duh \theta_{F,w;\leq n-1}(t^\prime) \nabla u_{n-1,\theta} \nabla w_{n,\theta} \dtp   \\
		&+ 2 \duh \theta_{F,w;\leq n-1}(t^\prime) \nabla P_{>N^\gamma} \nabla u_{n-1,\theta} \nabla F_{n,\theta}^\omega \dtp \qquad~.   \notag
\end{align}
Finally, we define \( u_{n,\theta} \) through the recursion
\begin{equation}\label{it:eq_truncated_un}
u_{n,\theta}= u_{n-1,\theta} + F_{n,\omega}^\omega + w_{n,\theta}~. 
\end{equation}
Since the truncations in \eqref{it:eq_truncated_Fn} and \eqref{it:eq_truncated_wn} depend on \( n \),  the function \( u_{n,\theta} \) no longer solves a (simple) differential equation. Once we remove the truncations, however, we will still obtain a solution of \eqref{intro:eq_un} on a random time interval.

\begin{rem}
In this section, we have carefully distinguished between the solutions of the actual and truncated differential equations. To simplify the notation, however, we will now drop the subscript \( \theta \). Unless stated otherwise, the functions \( F_n^\omega \), \( w_n \), and \( u_n \) are determined by \eqref{it:eq_truncated_Fn}, \eqref{it:eq_truncated_wn}, and \eqref{it:eq_truncated_un}. 
\end{rem}

\section{The adapted linear evolution \(F_n^\omega\)}\label{sec:lin}

In this section, we study the adapted linear evolution \( F_n^\omega \). Our main objective is to understand the frequency localization of the functions \( F_{n,k} \) and \( F_n^\omega \), which we then use to prove probabilistic Strichartz estimates. In order to avoid continually interrupting the main argument, we deal with any issues of (strong) measurability in the appendix.

\begin{prop}[Frequency profile of the adapted linear evolution] \label{lin:prop_freq_envelope}
Let \( (f_0,f_1) \in H_x^1 \times L_x^2 \). Let \( k \in \mathbb{Z}^3 \) with \( \| k \|_2 \sim N \), let \( \sigma > 1 \), let \( \rho:= \sigma-1-\delta >0 \), and let \( D^{\prime\prime} >0 \) be  arbitrarily large. Assume that \( \phi \colon \mathbb{R}^{1+3} \rightarrow \mathbb{R} \) has frequency support in the ball \( \| \xi \|_2 \lesssim N^\gamma \) and satisfies \( \langle \nabla \rangle^\sigma \phi \in L_t^1 L_x^\infty(\mathbb{R}\times \rthree)\). Furthermore, let \( F_k \) be the solution of 
\begin{equation}\label{lin:eq_Fk}
-\partial_{tt} F_k + \Delta F_k = 2 ~ \nabla \phi \cdot \nabla F_k~, \qquad (F_k,\partial_t F_k )|_{t=0}= ( \Pk f_0,\Pk f_1) ~. 
\end{equation}
Then, we have for all \( 0< T \leq 1 \) that 
\begin{align*}
&\| \nabla F_k \|_{L_t^\infty \Bkprime ([0,T]\times \rthree)}+ \| \partial_t F_k \|_{L_t^\infty \Bkprime ([0,T]\times \rthree)}+  \| F_k \|_{L_t^\infty \Bkprime ([0,T]\times \rthree)}\\
& \lesssim_{\sigma,\rho,\gamma,D^\prime} \| (\Pk f_0 , \Pk f_1) \|_{H^1\times L^2} \exp( C_{\sigma,\rho,\gamma,D^\prime} \| \langle \nabla \rangle^{\sigma}  \phi \|_{L_t^1 L_x^\infty([0,T]\times \rthree)} )~. 
\end{align*}
\end{prop}
\begin{proof} 
Let \( c= c^{\rho,\gamma}_{k,D^{\prime\prime}} \) be as in \eqref{prelim:eq_ck} and let \( \nabla_{x,t} \) be the gradient with respect to both variables. Then, we have that
\begin{equation}\label{lin:eq_Fk_switch}
\| \nabla_{x,t} F_k \|_{L_t^\infty \Bkprime} = \| c(M) \nabla_{x,t} P_{M;k} F_k \|_{L_t^\infty \ell_M^1 L_x^2}\leq 
\| c(M) \nabla_{x,t} P_{M;k} F_k \|_{\ell_M^1 L_t^\infty L_x^2}~. 
\end{equation}
Thus, we have to control \( \| \nabla_{x,t} P_{M;k} F_k \|_{L_t^\infty L_x^2} \). 
From Duhamels formula, it follows that
\begin{align*}
&\| P_{M;k} \nabla_{x,t} F_k \|_{L_t^\infty L_x^2([0,T]\times \rthree)} \\
&\lesssim \| P_{M;k} \nabla_{x,t} W(t) (\Pk f_0,\Pk f_1) \|_{L_t^\infty L_x^2([0,T]\times \rthree)} + \| P_{M;k} \left( \nabla F_k \cdot \nabla \phi \right) \|_{L_t^1 L_x^2([0,T]\times \rthree)} \\
&\lesssim 1_{M\leq 4} \| (\Pk f_0,\Pk f_1) \|_{\dot{H}^1\times L^2} + \| P_{M;k} \left( \nabla F_k \cdot \nabla \phi \right) \|_{L_t^1 L_x^2([0,T]\times \rthree)} 
\end{align*}
Then, we estimate
\begin{align*}
&\| P_{M;k} \left( \nabla \phi \cdot \nabla F_k \right) \|_{L_t^1 L_x^2([0,T]\times \rthree)} \\
&\lesssim \Big \| \sum_{K\ll M} \sum_{L\sim M} \| \nabla P_L \phi \cdot \nabla P_{K;k} F_k \|_{L_x^2} + \sum_{K\sim M} \| \nabla \phi  \cdot \nabla P_{K;k} F_k \|_{L_x^2} + \sum_{K\sim L \gg M } \| \nabla P_L \phi \cdot \nabla P_{K;k} F_k \|_{L_x^2} \Big\|_{L_t^1([0,T])} \\
&\lesssim 1_{\scriptscriptstyle M \lesssim N^\gamma}  M^{1-\sigma} (\sup_{K\ll M}  c(K)^{-1}) \Big \|   \| \langle \nabla \rangle^{\sigma} \phi \|_{L_x^\infty} \| \nabla F_k \|_{\Bkprime}   \Big \|_{L_t^1([0,T])}
+ \Big \| \|\langle \nabla \rangle^{\sigma} \phi \|_{L_x^\infty} \sum_{K\sim M} \| \nabla P_{K;k} F_k \|_{L_x^2}  \Big \|_{L_t^1([0,T])}\\
&~~+ 1_{\scriptscriptstyle{ M\lesssim N^\gamma}}  \sup_{K\gg M} (K^{1-\sigma} c(K)^{-1})  \Big \| \| \langle \nabla \rangle^\sigma \phi \|_{L_x^\infty} \| \nabla F_k \|_{\Bkprime} \Big \|_{L_t^1([0,T])}
\end{align*}
By multiplying with \( c(M) \), summing in \( M \), and interchanging \( \ell_M^1 \) and \( L_t^1 \) in the second contribution, we obtain that 
\begin{align*}
&\| \nabla F_k \|_{L_t^\infty \Bkprime([0,T])}+ \| \partial_t F_k \|_{L_t^\infty \Bkprime([0,T])} \\
&\lesssim \Big( 1+ \sum_{M\lesssim N^\gamma} M^{1-\sigma} c(M) \sup_{K\lesssim M} c(K)^{-1} + \sum_{M\lesssim N^\gamma} c(M) \sup_{K\gg M} K^{1-\sigma} c(K)^{-1} \Big)   \Big \| \| \langle \nabla \rangle^\sigma \phi \|_{L_x^\infty} \| \nabla F_k \|_{\Bkprime} \Big \|_{L_t^1([0,T])} \\
&\lesssim \Big( 1+ \sum_{M\lesssim N^\gamma} M^{1-\sigma+\rho} + \sum_{M\lesssim N^\gamma} M^{1-\sigma} \Big)   \Big \| \| \langle \nabla \rangle^\sigma \phi \|_{L_x^\infty} \| \nabla F_k \|_{\Bkprime} \Big \|_{L_t^1([0,T])}\\
&\lesssim   \Big \| \| \langle \nabla \rangle^\sigma \phi \|_{L_x^\infty} \| \nabla F_k \|_{\Bkprime} \Big \|_{L_t^1([0,T])}~.
\end{align*}
The proposition then follows from Gronwall's inequality. For the inhomogeneous term, we also use the fundamental theorem of calculus.\\
We remark that the definition of \( c(M) \) for \( M \gtrsim N^\gamma \) does not enter in a significant way. The weight only needs to grow in  \( M \) and satisfy a local constancy condition.
\end{proof}
\begin{cor} Under the same conditions as in Proposition \ref{lin:prop_freq_envelope}, we have that
\begin{equation*}
\begin{aligned}
&\| \nabla F_{n,k} \|_{\ell_k^2 L_t^\infty \Bkprime} + \| \partial_t F_{n,k} \|_{\ell_k^2 L_t^\infty \Bkprime}+ \|  F_{n,k} \|_{\ell_k^2 L_t^\infty \Bkprime}\\
& \lesssim \| (\widetilde{P}_N f_0, \widetilde{P}_N f_1) \|_{H^1\times L^2}  \exp( C_{\sigma,\rho,\gamma,D} \| \langle \nabla \rangle^{\sigma}  \phi \|_{L_t^1 L_x^\infty([0,T]\times \rthree)} )~. 
\end{aligned}
\end{equation*}
\end{cor}
\begin{proof}
This follows directly from Proposition \ref{lin:prop_freq_envelope} and \begin{equation*}
\| (\Pk f_0, \Pk f_1) \|_{\ell_{\| k\|_2 \sim N}^2 ( \dot{H}_x^1 \times L_x^2)} \lesssim \| (\widetilde{P}_N f_0, \widetilde{P}_N f_1) \|_{\dot{H}^1\times L^2}~. 
\end{equation*}
\end{proof}

Before we move on to the probabilistic Strichartz estimate, we also record the following estimate for \( F_n^\omega \).

\begin{cor}[Frequency localization of \( F_n^\omega \)] \label{lin:cor_frequency_Fnomega}
Let \( F_n^\omega \) be a solution of \eqref{it:eq_truncated_Fn}, let \( s \geq 1 \), and let \( D^\prime >0\). Then, we have that
\begin{equation}\label{lin:eq_frequency_Fnomega}
\| \langle \nabla \rangle^s F_n^\omega \|_{\XNprime([0,1])} + \| \langle \nabla \rangle^{s-1} \partial_t F_n^\omega \|_{\XNprime([0,1])} \lesssim \| (Q_N f_0^\omega, Q_N f_1^\omega ) \|_{H_x^s\times H_x^{s-1}} ~. 
\end{equation}
\end{cor}
\mbox{Corollary  \ref{lin:cor_frequency_Fnomega}} is a direct consequence of the work of Geba and Tataru, see \cite[Proposition 3.1]{TG05}. Since the principal symbol of \eqref{it:eq_truncated_Fn} has constant coefficients, we present a simpler and self-contained argument. We remark that the \( L_t^\infty\)-norm prevents us from using Khintchine's inequality, since this would lead to an \( N^{\epsilon}\)-loss, see \cite[Remark 3.8]{Bringmann18}.
\begin{proof} 
The proof relies on the energy method and Proposition \ref{lin:prop_freq_envelope}.
First, we prove the energy estimate
\begin{equation}\label{lin:eq_frequency_Fnomega_proof_1}
\| \nabla F_n^\omega \|_{L_t^\infty L_x^2} +\| \partial_t F_n^\omega \|_{L_t^\infty L_x^2} \lesssim \| (Q_N f_0^\omega, Q_N f_1^\omega) \|_{\dot{H}_x^1 \times L_x^2}~.
\end{equation}
Let \( \phi(t):= \theta_{F,w;\leq n-1}(t) P_{\leq N^\gamma} u_{n-1}(t) \).  Then, we have that 
\begin{equation*}
\frac{\mathrm{d}}{\mathrm{d}t} \frac{1}{2} \int_{\mathbb{R}^3} |\nabla F_n^\omega|^2 + (\partial_t F_n^\omega)^2 \dx = - 2 \int_{\mathbb{R}^3} \partial_t F_n^\omega~ \nabla \phi \cdot \nabla F_n^\omega \dx \leq \| \nabla \phi \|_{L_x^\infty} \int_{\mathbb{R}^3}  |\nabla F_n^\omega|^2 + (\partial_t F_n^\omega)^2 \dx~. 
\end{equation*}
The energy estimate \eqref{lin:eq_frequency_Fnomega_proof_1} then follows from the definition of \( \theta_{F,w;\leq n-1} \) and Gronwall's inequality. We now turn to the proof of \eqref{lin:eq_frequency_Fnomega}. 
For this, it suffices to show that 
\begin{equation}\label{lin:eq_frequency_Fnomega_proof_2}
c_{N,D^\prime}(M) \big( M^{s-1} \| \langle \nabla \rangle P_M F_n^\omega \|_{L_t^\infty L_x^2} + M^{s-1} \| \partial_t P_M F_n^\omega \|_{L_t^\infty L_x^2} \big)\lesssim N^{s-1} \| (Q_N f_0^\omega , Q_N f_1^\omega) \|_{H_x^1 \times L_x^2} ~. 
\end{equation}
If \( M \sim N \), then \eqref{lin:eq_frequency_Fnomega_proof_2} follows from \eqref{lin:eq_frequency_Fnomega_proof_1}. If \( M \not \sim N \), then \( \| k\|_2 \sim N \) implies that  
\begin{align*}
 &\| P_M \langle \nabla \rangle F_{n,k}\|_{L_t^\infty L_x^2} + \| P_M \partial_t F_n^\omega \|_{L_t^\infty L_x^2} \\
 & \lesssim \left(\frac{\max(N,M)}{N^\gamma} \right)^{-D^{\prime\prime}} \left( \| \langle \nabla \rangle F_{n,k} \|_{L_t^\infty \Bkprime ([0,T]\times \rthree)}+  \| \partial_t F_{n,k} \|_{L_t^\infty \Bkprime ([0,T]\times \rthree)} \right)
\end{align*}
By choosing \( D^{\prime\prime} > 0 \) large enough, it follows from Proposition \ref{lin:prop_freq_envelope} that
\begin{align*}
&\| P_M \langle \nabla \rangle F_n^\omega \|_{L_t^\infty L_x^2} + \| P_M \partial_t F_n^\omega \|_{L_t^\infty L_x^2} \\
&\leq \sum_{N/2 \leq \| k \|_2 < N} |g_k| \left( \| P_M \langle \nabla \rangle F_{n,k}\|_{L_t^\infty L_x^2} + \| P_M \partial_t F_{n,k} \|_{L_t^\infty L_x^2} \right) \\
&\lesssim (MN)^{-4D^\prime} \hspace{-2ex} \sum_{N/2\leq \| k\|_2 < N} |g_k| \big( \| \langle \nabla \rangle F_{n,k} \|_{L_t^\infty \Bkprime ([0,T]\times \rthree)}+ \| \partial_t F_{n,k} \|_{L_t^\infty \Bkprime ([0,T]\times \rthree)} \big)\\
&\lesssim (MN)^{-4D^\prime} \sum_{N/2\leq \| k\|_2 < N}  |g_k|  \| (\Pk f_0 , \Pk f_1) \|_{H^1_x\times L^2_x}\\
&\lesssim (MN)^{-4D^\prime} N^{\frac{3}{2}} \Big( \sum_{N/2\leq \| k\|_2 < N}  |g_k|^2  \| (\Pk f_0 , \Pk f_1) \|_{H^1_x\times L^2_x}^2\Big)^{\frac{1}{2}} \\
&\lesssim (MN)^{-2D^\prime} \| (Q_N f_0^\omega , Q_N f_1^\omega) \|_{H_x^1 \times L_x^2} ~. 
\end{align*}
This estimate is stronger than \eqref{lin:eq_frequency_Fnomega_proof_2}, and it completes the proof. 
\end{proof}

\begin{prop}[Probabilistic Strichartz Estimates]\label{lin:prop_probabilistic_strichartz}
Let \( F_n^\omega \) be a solution of \eqref{it:eq_truncated_Fn}, let \( s>1 \), \( \sigma^\prime > \sigma > 1 \), and let \( D^\prime > 0 \).  Let \( \delta > 0 \) be as in Proposition \ref{lin:prop_freq_envelope}. Furthermore, we assume that \begin{equation}\label{lin:eq_consistency_condition_strichartz}
\sigma < \frac{3}{2}~. 
\end{equation}
Then, it holds for all \( 0 < T \leq 1 \) and all \( r \geq 1 \) that
\begin{equation}\label{lin:eq_probabilistic_strichartz}
\begin{aligned}
&\| \langle\nabla\rangle^{\sigma^\prime} F_n^\omega \|_{L_\omega^r \SNprime(\Omega \times [0,T])} \\
&\lesssim \sqrt{r} T^\frac{1}{2} N^{2\delta} \| (\widetilde{P}_N f_0, \widetilde{P}_N f_1) \|_{H_x^{\sigma^\prime}(\rthree)\times H_x^{\sigma^\prime-1}(\rthree)} \\
&+ \sqrt{r} T^{\frac{1}{2}} N^{\sigma^\prime -s +1- \gamma (\sigma-1) - \frac{1}{2} (1-\gamma) + 2\delta }  \| (\widetilde{P}_N f_0, \widetilde{P}_N f_1) \|_{H_x^s(\rthree)\times H_x^{s-1}(\rthree)}
\end{aligned}
\end{equation}
\end{prop}

\begin{rem}
The power on \( N \) in the estimate above can be motivated by writing
\begin{equation*}
N^{\sigma^\prime -s +1- \gamma (\sigma-1) - \frac{1}{2} (1-\gamma) }= \underbrace{N^{\sigma^\prime-s}}_{\substack{\text{difference of}\\\text{derivatives}}} \cdot \underbrace{N^1}_{\substack{\text{deterministic}\\ \text{Strichartz}}} \cdot \underbrace{{N^{-\gamma(\sigma-1)}}}_{\substack{\text{gain frequency}\\ \text{localization}}} \cdot \underbrace{N^{-\frac{1}{2}(1-\gamma)}}_{\substack{\text{gain refined}\\\text{Strichartz}}}~.
\end{equation*}
The nearly optimal choice of the parameters leads to \( \sigma^\prime = 1.13205 \), see \eqref{fin:eq_parameters}. From \eqref{fin:eq_parameters}, we also have that \( {\sigma= \nu-1-=1.1001-} \), and thus the random evolution \( F_n^\omega \) has a higher number of derivatives bounded in \( L_t^2 L_x^\infty \) than \( w_n \).
\end{rem}
\begin{proof}[Proof of Proposition \ref{lin:prop_probabilistic_strichartz}:] Let \( (q,p) =(2+,\infty-) \) be a sharp wave-admissible Strichartz pair. During this proof, it is convenient to define
\begin{equation*}
\| u \|_{\SNqpprime([0,T])} := \sum_{M\geq 1} c_{N,D^\prime}(M) ~  \| P_M u \|_{L_t^q L_x^p([0,T]\times \rthree)}~. 
\end{equation*}

We separate the proof in three steps: \\

\textbf{Step 1: Estimate for the individual \( F_{n,k} \).}~ \\
Since \( F_n^\omega \) is a random linear combination of the \( F_{n,k} \), we need to control their Strichartz-type norms. To this end, let \( F_k \) be a solution of \eqref{lin:eq_Fk} and assume that \( \phi \) has frequency support inside the ball \( \| \xi \|_2 \leq N^\gamma \). We prove that for any \( D^\prime > 0 \) there exists a \(D^{\prime\prime}>0 \) s.t.
\begin{equation}\label{lin:eq_refined_strichartz}
\begin{aligned}
\| \langle \nabla \rangle^{\sigma^\prime} F_k \|_{\SNqpprime([0,T])}
 &\lesssim T^{\frac{1}{q}}  \| (\Pk f_0,\Pk f_1) \|_{H_x^{\sigma^\prime} \times H_x^{\sigma^\prime-1}} \\
 &+  N^{\sigma^\prime-\gamma(\sigma-1)-\frac{1}{2}(1-\gamma)+\delta }  \| \langle \nabla \rangle^\sigma \phi \|_{L_t^1 L_x^\infty( \sIT)} \| \nabla F_k\|_{L_t^\infty \Bktil ([0,T])}~.
\end{aligned}
\end{equation}
Using Hölder's inequality and Bernstein's estimate, we have that
\begin{equation*}
\| \langle \nabla \rangle^{\sigma^\prime} P_M W(t) ( \Pk f_0, \Pk f_1) \|_{L_t^q L_x^p}  
\lesssim T^{\frac{1}{q}} \IMN   \| (\Pk f_0, \Pk f_1) \|_{H_x^{\sigma^\prime}\times H_x^{\sigma^\prime-1}}~. 
\end{equation*}
This leads to the first summand on the right-hand side of \eqref{lin:eq_refined_strichartz}. Next, we control the Duhamel term. By using a Littlewood-Paley decomposition, we obtain that
\begin{align*}
&\| \langle \nabla \rangle^{\sigma^\prime} P_M \int_0^t \frac{\sin((t-t^\prime)|\nabla|)}{|\nabla|} \nabla \phi \cdot \nabla F_k \dtp \|_{L_t^q L_x^p} \\
&\lesssim \sum_{L \leq N^\gamma} \sum_{K \ll N} \| \langle \nabla \rangle^{\sigma^\prime} P_M \int_0^t \frac{\sin((t-t^\prime)|\nabla|)}{|\nabla|} \nabla P_L \phi \cdot \nabla P_{K;k} F_k \dtp \|_{L_t^q L_x^p} \\
&~~+ \sum_{K\gtrsim N} \| \langle \nabla \rangle^{\sigma^\prime} P_M \int_0^t \frac{\sin((t-t^\prime)|\nabla|)}{|\nabla|} \nabla  \phi \cdot \nabla  P_{K;k} F_k \dtp \|_{L_t^q L_x^p}~. 
\end{align*}
We first control the contribution of the main term \( K \ll N \). Using the refined Strichartz estimates (Lemma \ref{prelim:lem_refined_strichartz}), it holds that 
\begin{align*}
&\sum_{ L \leq N^\gamma} \sum_{K \ll N} \| \langle \nabla \rangle^{\sigma^\prime} P_M \int_0^t \frac{\sin((t-t^\prime)|\nabla|)}{|\nabla|} \nabla P_L \phi \cdot \nabla P_{K;k} F_k \dtp \|_{L_t^q L_x^p}\\
&\lesssim M^{\sigma^\prime} \sum_{L \leq N^\gamma} \sum_{K \ll N}  \bigg( \frac{\max(L,K)}{N} \bigg)^{\frac{1}{2}-} \| P_M( \nabla P_L \phi \cdot \nabla P_{K;k} F_k) \|_{L_t^1 L_x^2} \\
&\lesssim 1_{M\sim N} M^{\sigma^\prime}  \sum_{L \leq N^\gamma} \sum_{K \ll N} \bigg( \frac{\max(L,K)}{N} \bigg)^{\frac{1}{2}-}  \| \nabla P_L \phi \|_{L_t^1 L_x^\infty} \| \nabla P_{K;k} F_k \|_{L_t^\infty L_x^2} \\
&\lesssim 1_{M\sim N} N^{\sigma^\prime-\frac{1}{2}+} \bigg( \sum_{L \leq N^\gamma} \sum_{K \ll N}   \max(L,K)^{\frac{1}{2}}  L^{1-\sigma} K^{1-\sigma+\delta} \max\Big( 1, \frac{K}{N^\gamma}\Big)^{-D^{\prime\prime}} \bigg) \| \langle \nabla \rangle^\sigma \phi \|_{L_t^1 L_x^\infty}\| \nabla F_k\|_{L_t^\infty \Bktil}
\end{align*}
To complete the estimate for \( K \ll N \), it only remains to evaluate the sum in \( L \) and \( K \). We have that
\begin{align*}
&N^{\sigma^\prime-\frac{1}{2}+}  \sum_{L \leq N^\gamma} \sum_{K \ll N}    L^{1-\sigma} K^{1-\sigma+\delta} \\
&\lesssim N^{\sigma^\prime-\frac{1}{2}+} \sum_{L,K \leq N^\gamma }  L^{1-\sigma} K^{\frac{3}{2}-\sigma+\delta} + N^{\sigma^\prime-\frac{1}{2}+} \bigg(\sum_{L \leq N^\gamma}L^{1-\sigma} \bigg)\cdot \bigg( \sum_{N^\gamma \leq K \ll N} K^{\frac{3}{2}-\sigma+\delta} \bigg( \frac{K}{N^\gamma} \bigg)^{-D^{\prime\prime}}  ~ \bigg)\\
&\lesssim N^{\sigma^\prime- \gamma(1-\sigma) - \frac{1}{2} (1-\gamma)+\delta}~. 
\end{align*}
We now control the contribution for \( K \gtrsim N \). Since \( D^{\prime\prime} \) can be chosen sufficiently large, it holds that 
\begin{align*}
&\sum_{K\gtrsim N} \| \langle \nabla \rangle^{\sigma^\prime} P_M \int_0^t \frac{\sin((t-t^\prime)|\nabla|)}{|\nabla|} \nabla  \phi \cdot \nabla  P_{K;k} F_k \dtp \|_{L_t^q L_x^p} \\
&\lesssim M^{\sigma^\prime} \sum_{K\gtrsim N} \| P_M ( \nabla \phi \cdot \nabla P_{K;k} F) \|_{L_t^1 L_x^2} \\
&\lesssim M^{\sigma^\prime} \sum_{K\gtrsim \max(N,M)}  \| \nabla \phi \|_{L_t^1 L_x^\infty} \| \nabla P_{K;k} F_k \|_{L_t^\infty L_x^2} \\
&\lesssim M^{\sigma^\prime}  \bigg(\sum_{K\gtrsim \max(N,M)} K^{1-\sigma+\delta} \bigg( \frac{K}{N^\gamma} \bigg)^{-D^{\prime\prime}} \bigg) \| \nabla \phi \|_{L_t^1 L_x^\infty} \| \nabla F_k\|_{L_t^\infty \Bktil}\\
&\lesssim (NM)^{-10 (D^\prime+1)} \| \langle \nabla \rangle^{\sigma} \phi \|_{L_t^1 L_x^\infty}  \| \nabla F_k\|_{L_t^\infty \Bktil} ~. 
\end{align*}
This completes the proof of \eqref{lin:eq_refined_strichartz}, which we now apply to the functions \( F_{n,k} \).
Due to the cutoff, \( \sigma^\prime > \sigma\), and \eqref{it:eq_triangle}, we have that
\begin{equation*}
\| \langle \nabla \rangle^{\sigma} \left( \theta_{F,w;\leq n-1}(t) u_{n-1}(t,x) \right) \|_{L_t^2 L_x^\infty(\mathbb{R}\times \rthree)} \lesssim 1~. 
\end{equation*}
Thus, it follows from \eqref{lin:eq_refined_strichartz} and  Proposition \ref{lin:prop_freq_envelope} that
\begin{equation}\label{lin:eq_refined_strichartz_Fnk}
\begin{aligned}
&\| \langle \nabla \rangle^{\sigma^\prime} F_{n,k} \|_{\SNqpprime} \\
&\lesssim  T^\frac{1}{q} \| (\Pk f_0, \Pk f_1) \|_{H_x^{\sigma^\prime}\times H_x^{\sigma^\prime-1}} +  T^{\frac{1}{2}}N^{\sigma^\prime-\frac{1}{2} (1-\gamma) + \gamma (1-\sigma) + \delta + } \| \nabla F_{n,k} \|_{L_t^\infty \Bktil([0,T])}  \\
&\lesssim T^\frac{1}{q} \| (\Pk f_0, \Pk f_1) \|_{H_x^{\sigma^\prime}\times H_x^{\sigma^\prime-1}} 
+  T^{\frac{1}{2}} N^{\sigma^\prime-\frac{1}{2} (1-\gamma) + \gamma (1-\sigma) + \delta + } \| ( \Pk f_0 , \Pk f_1 ) \|_{H_x^1\times L_x^2}  \\
&\lesssim T^\frac{1}{q} \| (\Pk f_0, \Pk f_1) \|_{H_x^{\sigma^\prime}\times H_x^{\sigma^\prime-1}} 
+  T^{\frac{1}{2}} N^{\sigma^\prime-s+1-\frac{1}{2} (1-\gamma) + \gamma (1-\sigma) + \delta + } \| ( \Pk f_0 , \Pk f_1 ) \|_{H_x^s\times H_x^{s-1}} 
\end{aligned}
\end{equation}

\textbf{Step 2: Probabilistic Decoupling in \( \SNqpprime\).} ~\\
In this step, we use \eqref{lin:eq_refined_strichartz_Fnk} to prove the analog of \eqref{lin:eq_probabilistic_strichartz} in \( \SNqpprime \). More precisely, we prove that 
\begin{equation} \label{lin:eq_prob_strichartz_qp}
\begin{aligned}
\| \langle \nabla \rangle^{\sigma^\prime} F_n^\omega \|_{L_\omega^r \SNqpprime} &\lesssim \sqrt{r} T^\frac{1}{q} \|(\widetilde{P}_N f_0, \widetilde{P}_N f_1)\|_{H_x^{\sigma^\prime}(\rthree)\times H_x^{\sigma^\prime-1}(\rthree)} \\
&~~+  \sqrt{r} T^{\frac{1}{2}} N^{\sigma^\prime-s+1-\frac{1}{2} (1-\gamma) + \gamma (1-\sigma) + \delta + } \| (\widetilde{P}_N f_0, \widetilde{P}_N f_1) \|_{H_x^s(\rthree)\times H_x^{s-1}(\rthree)} ~.
\end{aligned}
\end{equation}
We use a standard combination of Khintchine's inequality and Minkowski's integral inequality to extend the estimate from \( F_{n,k} \) to  \( F_n^\omega \), see e.g. \cite{BOP15,LM16}.

In the equations below, we let \( k\in \mathbb{Z}^3 \) be in the annulus
\( N/2 \leq \| k \|_2 < N \). Recall that the functions \( F_{n,k} \) are measurable with respect to the sigma-algebra \( \mathscr{F}_{n-1}  = \sigma( g_l \colon \| l \|_2 < N/2 ) \), and that the random variables \( \{ g_k \colon N/2\leq \|k\|_2 < N \} \) are independent of \( \mathscr{F}_{n-1} \). By conditioning on \( \mathscr{F}_{n-1} \), we obtain that

\begin{equation*}
\| \langle \nabla \rangle^{\sigma^\prime} F_n^\omega \|_{L_\omega^r \SNqpprime}= \mathbb{E}\Big[ \mathbb{E} \Big[ \| c_{N,D^\prime}(M) \sum_{k} g_k \langle \nabla \rangle^{\sigma^\prime} P_M F_{n,k} \|_{\ell_M^1 L_t^q  L_x^p}^r\Big| \mathscr{F}_{n-1} \Big] \Big] ^{\frac{1}{r}} 
\end{equation*}
From Minkowski's integral inequality and Khintchine's inequality, we obtain for all \( r \geq \max(q,p) \) that
\begin{align*}
& \mathbb{E}\Big[ \mathbb{E} \Big[ \| c_{N,D^\prime}(M) \sum_{k} g_k \langle \nabla \rangle^{\sigma^\prime} P_M F_{n,k} \|_{\ell_M^1 L_t^q  L_x^p}^r\Big| \mathscr{F}_{n-1} \Big] \Big] ^{\frac{1}{r}} \notag\\
&\leq \mathbb{E}\Big[ \Big\| \mathbb{E} \Big[ \big| c_{N,D^\prime}(M) \sum_{k} g_k \langle \nabla \rangle^{\sigma^\prime} P_M F_{n,k} \big|^r \Big| \mathscr{F}_{n-1}\Big]^{\frac{1}{r}} \Big\|_{\ell_M^1 L_t^q  L_x^p}^r \Big] ^{\frac{1}{r}} \notag\\
& \lesssim\sqrt{r}  \mathbb{E} \Big[ \| c_{N,D^\prime}(M)  \langle \nabla \rangle^{\sigma^\prime} P_M F_{n,k} \|_{ \ell_M^1 L_t^q L_x^p\ell_k^2}^r \Big] ^{\frac{1}{r}} \notag\\
&=\sqrt{r} \| c_{N;D^\prime}(M) \langle \nabla \rangle^{\sigma^\prime} P_M F_{n,k} \|_{L_\omega^r \ell_M^1 L_t^q L_x^p \ell_k^2}
\end{align*}
In order to use Minkowski's integral inequality again, we need to move from \( \ell_M^1 \) to \( \ell_M^2 \). Using \eqref{lin:eq_refined_strichartz_Fnk} with \( D^\prime+1\) instead of \( D^\prime \), we may increase the weight from \( c_{N;D^\prime} \) to \( c_{N;D^\prime+1} \). Then, it follows that
\begin{align*}
&\sqrt{r} \| c_{N;D^\prime}(M) \langle \nabla \rangle^{\sigma^\prime} P_M F_{n,k} \|_{L_\omega^r \ell_M^1 L_t^q L_x^p \ell_k^2}\\
&\lesssim \sqrt{r} \| c_{N,D^\prime+1}(M) \langle \nabla \rangle^{\sigma^\prime} P_M F_{n,k} \|_{L_\omega^r \ell_M^2 L_t^q  L_x^p\ell_k^2} \\
&\lesssim \sqrt{r} \| c_{N,D^\prime+1}(M) \langle \nabla \rangle^{\sigma^\prime} P_M F_{n,k} \|_{\ell_k^2 L_\omega^r\ell_M^2 L_t^q  L_x^p} \\
& \lesssim\sqrt{r}  \|   \langle \nabla \rangle^{\sigma^\prime} F_{n,k} \|_{\ell_k^2 L_\omega^r  \SNqpplus} \notag\\
&\lesssim \sqrt{r} T^\frac{1}{q} \|(\widetilde{P}_N f_0, \widetilde{P}_N f_1)\|_{H_x^{\sigma^\prime}(\rthree)\times H_x^{\sigma^\prime-1}(\rthree)} \notag\\
&~~+  \sqrt{r} T^{\frac{1}{2}} N^{\sigma^\prime-s+1-\frac{1}{2} (1-\gamma) + \gamma (1-\sigma) + \delta + } \| (\widetilde{P}_N f_0, \widetilde{P}_N f_1) \|_{H_x^s(\rthree)\times H_x^{s-1}(\rthree)} ~.\notag
\end{align*}
The same estimate for \(1\leq r\leq \max(q,p) \) then follows by using Hölder's inequality in \( \omega \). This completes the proof of \eqref{lin:eq_prob_strichartz_qp}.\\

\textbf{Step 3: Moving from \((2,\infty)\) to \( (q,p) \).} ~\\
Using Bernstein's estimate, we have that
\begin{align*}
\| \langle \nabla \rangle^{\sigma^\prime} F_n^\omega \|_{\SNprime}
&\leq \sum_{L\geq 1} c_{N,D^\prime}(L) \| \langle \nabla \rangle^{\sigma^\prime} P_L F_n^\omega \|_{L_t^2 L_x^\infty} \\
&\lesssim T^{\frac{1}{2}-\frac{1}{q}} \sum_{L\geq 1} L^{0+} c_{N,D^\prime}(L)  \| \langle \nabla \rangle^{\sigma^\prime} P_L F_n^\omega \|_{L_t^q L_x^p}\\
&\lesssim T^{\frac{1}{2}-\frac{1}{q}} \sup_{L\geq 1} \Big(L^{0+} \max \Big( \frac{N}{L}, \frac{L}{N}\Big)^{-1} \Big) \| \langle \nabla \rangle^{\sigma^\prime} F_n^\omega \|_{\SNqpplus} \\
&\lesssim  T^{\frac{1}{2}-\frac{1}{q}} N^{0+} \| \langle \nabla \rangle^{\sigma^\prime} F_n^\omega \|_{\SNqpplus}~. 
\end{align*}
Then, the proposition follows from \eqref{lin:eq_prob_strichartz_qp}, where \( D^\prime \) is replaced by \( D^\prime+1\).
\end{proof}

\section{The nonlinear evolution \( w_n \)}
Recall that the nonlinear evolution \( w_n \) solves the truncated equation
\label{nl:eq_truncated_wn}
\begin{align}
w_n(t) &= \duh  \theta_{F;n}(t^\prime) |\nabla F_n^\omega|^2\dtp  \notag \\
		&~~~+ 2 \duh \theta_{F;n}(t^\prime) \nabla F_n^\omega \nabla w_n \dtp \notag\\
		&~~~+ \duh \theta_{w;n}(t^\prime) |\nabla w_n|^2\dtp \notag\\
		&~~~+ 2 \duh \theta_{F,w;\leq n-1}(t^\prime) \nabla F_{\leq n-1}^\omega \nabla w_n \dtp \label{nl:eq_truncated_wn}\\
		&~~~+ 2 \duh \theta_{F,w;\leq n-1}(t^\prime) \nabla w_{\leq n-1} \nabla w_n \dtp \notag \\
		&~~~+ 2 \duh \theta_{F,w;\leq n-1}(t^\prime) \nabla P_{>N^\gamma} \nabla u_{n-1}\nabla F_n^\omega \dtp~\notag \\
\end{align}
The main result of this section provides control of the nonlinear component \( w_n \) in \( \YN \). 

\begin{prop}[Control of the nonlinear component \( w_n \)]\label{nl:prop_wn}
Assume that \( \nu> 2 \), \( \sigma=\nu-1- \), \( s> 1 \), \( \sigma^\prime > 1 \), and \( \max( \nu-\sigma^\prime, \sigma-1 ) < \eta < \nu - 1 \). Let \( D\geq D_0(s,\nu,\sigma^\prime,\sigma,\eta) \) and \( D^\prime\geq D_0^\prime(s,\nu,\sigma^\prime,\sigma,\eta,D) \) be sufficiently large. Furthermore, assume that \( 0<T_0 = T_0(s,\nu,\sigma^\prime,\sigma,\eta,D,D^\prime) \) is sufficiently small. \\
Then there exists a unique solution \( w_n \in \YN([0,T_0]) \) of  \eqref{nl:eq_truncated_wn}. Furthermore, we have for all \( 0 \leq T \leq T_0 \) that
\begin{equation}\label{nl:eq_a_priori}
\begin{aligned}
\| w_n \|_{\YN([0,T])}
&\lesssim T^{\frac{1}{2}} \left( N^{\nu-s-\gamma(\sigma^\prime-1)}+ N^{(1-\gamma)(\nu-1)+1-\sigma^\prime} \right) \left( \| \langle \nabla \rangle^s F_n^\omega \|_{\XNprime}+ \| \langle \nabla \rangle^{\sigma^\prime} F_n^\omega \|_{\SNprime}\right)~.
\end{aligned}
\end{equation}
\end{prop}

\subsection{Bilinear Estimates}
In this section we prove the main bilinear estimates for the Duhamel terms in \eqref{nl:eq_truncated_wn}. In order to group similar estimates together, we work with a paraproduct decomposition. We define
\begin{align*}
\Plh(v,w) &:= \sum_{L,K\colon L \ll K} \duh \nabla P_L v \cdot \nabla P_K w ~\dtp ~,\\
\Phl(v,w) &:=\sum_{L,K\colon L \gg K } \duh \nabla P_L v \cdot \nabla P_K w ~\dtp ~, \\
\Phh(v,w) &:= \sum_{L, K \colon L \sim K} \duh \nabla P_L \cdot \nabla P_K w ~\dtp ~.  
\end{align*}
Our motivation for distinguishing between low-high and high-low interactions stems from the terms in \eqref{nl:eq_truncated_wn}. Whereas the first factor is often localized at frequencies \( \lesssim N\), the second factor is always localized at frequencies \( \sim N \).\\
We now summarize the necessary estimates for the proof of Proposition \ref{nl:prop_wn}. The functions \( F,G \) below correspond to either \( F_{\leq n-1}^\omega \) or \( F_n^\omega\), and the functions \( v,w \) below correspond to either \( w_{\leq n-1} \) or \( w_n \). To simplify the notation, recall from Section \ref{sec:function_spaces} that
\begin{equation*}
\begin{aligned}
\| u \|_{\YN} = \| u \|_{\YN([0,T])}&:= \| \langle \nabla \rangle^\nu u \|_{(\XNeta \bigcap \XlN)([0,T])} + \| \langle \nabla \rangle^{\nu-1} \partial_t u \|_{(\XNeta \bigcap \XlN)([0,T])} \\
							 &~+\| \langle \nabla \rangle^{\sigma} u \|_{(\SNeta \bigcap \SlN)([0,T])}~. 
\end{aligned}
\end{equation*}

\begin{lem}[Low-High Interactions]\label{nl:lem_low_high}
Assume that \( \nu> 2 \), \( \sigma=\nu-1- \), \( s> 1 \), \( \sigma^\prime > 1 \), and \( \eta >0  \). Let \( D\geq D_0(s,\nu,\sigma^\prime,\sigma,\eta) \) and \( D^\prime\geq D_0^\prime(s,\nu,\sigma^\prime,\sigma,\eta,D) \) be sufficiently large. Then, we have for any \( 0 < T \leq 1 \) that  
\begin{align}
\|  \Plh(G,F) \|_{\YN}  &\lesssim T^{\frac{1}{2}} N^{\nu-s+1-\sigma^\prime} \| \langle \nabla \rangle^{\sigma^\prime} G \|_{\SNprime} \| \langle \nabla \rangle^s F \|_{\XNprime}~, \label{nl:eq_low_high_FF}\\
\|  \Plh(P_{>N^\gamma} G,F) \|_{\YN} &\lesssim T^{\frac{1}{2}} N^{\nu-s+\gamma(1-\sigma^\prime)} \| \langle \nabla \rangle^{\sigma^\prime} G \|_{\SlNprime} \| \langle \nabla \rangle^s F \|_{\XNprime}~, \label{nl:eq_low_high_PFF}~,\\
\| \Plh(P_{>N^\gamma} v , F) \|_{\YN} & \lesssim T^{\frac{1}{2}} N^{(1-\gamma)(\nu-1)+1-\sigma^\prime} \| \langle \nabla \rangle^{\nu} v \|_{\XlN} \| \langle \nabla \rangle^{\sigma^\prime} F \|_{\SNprime} \label{nl:eq_low_high_PwF}~, \\
\| \Plh(G,w) \|_{\YN} &\lesssim  T^{\frac{1}{2}} \| \langle \nabla \rangle^{\sigma^\prime} G \|_{\SlNprime} \| \langle \nabla \rangle^{\nu} w \|_{\XNeta \bigcap \XlN} \label{nl:eq_low_high_Fw}\\
\| \Plh(v,w) \|_{\YN} &\lesssim T^{\frac{1}{2}} \| \langle \nabla \rangle^{\sigma} v\|_{\SlN} \| \langle \nabla \rangle^{\nu} v \|_{\XNeta \bigcap \XlN}~. \label{nl:eq_low_high_ww}
\end{align}
\end{lem}

\begin{lem}[High-Low Interactions]\label{nl:lem_high_low}
Assume that \( \nu> 2 \), \( \sigma=\nu-1- \), \( s> 1 \), \( \sigma^\prime > 1 \), and \( \max( \nu-\sigma^\prime, \sigma-1 ) < \eta < \nu - 1 \). Let \( D\geq D_0(s,\nu,\sigma^\prime,\sigma,\eta) \) and \( D^\prime\geq D_0^\prime(s,\nu,\sigma^\prime,\sigma,\eta,D) \) be sufficiently large. Then, we have for any \( 0 < T \leq 1 \) that  
\begin{align}
\| \Phl(G,F) \|_{\YN} &\lt N^{\nu-s+1-\sigma^\prime} \|\langle \nabla \rangle^{\sigma^\prime} G \|_{\SlNprime} \| \langle \nabla \rangle^s F\|_{\XNprime}~\label{nl:eq_high_low_FF}~, \\
\| \Phl(v,F) \|_{\YN} &\lt N^{1-\sigma^\prime} \| \langle \nabla \rangle^{\nu} v \|_{\XlN} \| \langle \nabla \rangle^{\sigma^\prime} F \|_{\SNprime} \label{nl:eq_high_low_wF}~,\\
\| \Phl(G,w) \|_{\YN} &\lt  \| \langle \nabla \rangle^{\sigma^\prime} G \|_{\SlNprime} \| \langle \nabla \rangle^\nu w \|_{\XNeta}~, \label{nl:eq_high_low_Fw}\\
\| \Phl(v,w) \|_{\YN} &\lt \| \langle \nabla \rangle^{\nu}  v \|_{\XlN} \| \langle \nabla \rangle^{\sigma} w \|_{\SNeta}~. \label{nl:eq_high_low_ww}
\end{align}
\end{lem}

\begin{lem}[High-High Interactions]\label{nl:lem_high_high}
Assume that \( \nu> 2 \), \( \sigma=\nu-1- \), \( s> 1 \), \( \sigma^\prime > 1 \), and \( 0 < \eta < \nu - 1 \).  Let \( D\geq D_0(s,\nu,\sigma^\prime,\sigma,\eta) \) and \( D^\prime\geq D_0^\prime(s,\nu,\sigma^\prime,\sigma,\eta,D) \) be sufficiently large. Then, we have for any \( 0 < T \leq 1 \) that  
\begin{align}
\| \Phh(G,F) \|_{\YN} &\lt N^{\nu-s+1-\sigma^\prime} \| \langle \nabla \rangle^{\sigma^\prime} G \|_{\SlN} \| \langle \nabla \rangle^s F \|_{\XNprime}~, \label{nl:eq_high_high_FF}\\
\| \Phh(v,F) \|_{\YN} &\lt N^{1-\sigma^\prime} \| \langle \nabla \rangle^{\nu} v \|_{\XlN} \| \langle \nabla \rangle^{\sigma^\prime} F \|_{\SNprime} ~, \label{nl:eq_high_high_wF} \\
\| \Phh(G,w) \|_{\YN} &\lt \| \langle \nabla \rangle^{\sigma^\prime} G \|_{\SlNprime} \| \langle \nabla\rangle^{\nu} w \|_{\XNeta}~, \label{nl:eq_high_high_Fw} \\
\| \Phh(v,w) \|_{\YN} &\lt \| \langle \nabla \rangle^{\sigma} v \|_{\SlN} \| \langle \nabla \rangle^{\nu} w \|_{\XNeta \bigcap \XlN} \label{nl:eq_high_high_ww}~.  
\end{align}
\end{lem}
Since the (standard) proofs of the inequalities \eqref{nl:eq_low_high_FF}-\eqref{nl:eq_high_high_ww} are relatively long, we postpone them until Section \ref{section:proof_bilinear}. We now use the estimates above to control the contribution of \( \nabla P_{>N^\gamma} u_{n-1} \cdot \nabla F_n^\omega \). Under certain conditions on the parameters, this term will be smoother than the adapted linear evolution \( F_n^\omega \). This shows that we removed the unfavorable low-high interaction described in the introduction. Since the low-high interaction is the principal obstacle in the control of the nonlinear component \( w_n \), this is the main step in the proof of Proposition \ref{nl:prop_wn}. 
\begin{cor}[Control of \( \nabla P_{>N^\gamma} u_{n-1}\cdot \nabla F_n^\omega \)]\label{nl:cor_inhomogeneous_contribution}
Assume that \( \nu> 2 \), \( \sigma=\nu-1- \), \( s> 1 \), \( \sigma^\prime > 1 \), and \( \max( \nu-\sigma^\prime, \sigma-1 ) < \eta < \nu - 1 \). Let \( D\geq D_0(s,\nu,\sigma^\prime,\sigma,\eta) \) and \( D^\prime\geq D_0^\prime(s,\nu,\sigma^\prime,\sigma,\eta,D) \) be sufficiently large. Then, we have for any \( 0 < T \leq 1 \) that  

\begin{equation}\label{nl:eq_inhomogeneous_contribution}
\begin{aligned}
&\Big\| \duh \theta_{F,w;\leq n-1}(t^\prime) \nabla P_{>N^\gamma} u_{n-1} \nabla F_n^\omega \dtp \Big\|_{\YN} \\
&\lt \left( N^{\nu-s+\gamma(1-\sigma^\prime)} + N^{(1-\gamma) (\nu-1)+1-\sigma^\prime} \right) \left( \| \bra^s F_n^\omega \|_{\XNprime} + \| \bra^{\sigma^\prime} F_n^\omega \|_{\SNprime} \right)~. 
\end{aligned}
\end{equation}
\end{cor}
\begin{proof}
We split \( u_{n-1} =  F_{\leq n-1}^\omega + w_{\leq n-1} \). \\
Using \eqref{nl:eq_low_high_PFF}, \eqref{nl:eq_high_low_FF}, and \eqref{nl:eq_high_high_FF}, we have that 
\begin{align*}
&\Big \| \duh \theta_{F,w;\leq n-1}(t^\prime) \nabla P_{>N^\gamma} F_{\leq n-1}^\omega \nabla F_n^\omega ~ \dtp \Big \|_{\YN} \\
&\lt N^{\nu-s+\gamma(1-\sigma^\prime)} \| \theta_{F,w;\leq n-1} \langle \nabla \rangle^{\sigma^\prime} F_{\leq n-1}^\omega \|_{\SlNprime} \| \bra^s F_n^\omega \|_{\XNprime} \\
&\lt N^{\nu-s+\gamma(1-\sigma^\prime)}\| \bra^s F_n^\omega \|_{\XNprime} ~. 
\end{align*}
Using \eqref{nl:eq_low_high_PwF}, \eqref{nl:eq_high_low_wF}, and \eqref{nl:eq_high_high_wF}, we have that 
\begin{align*}
&\Big \| \duh \theta_{F,w;\leq n-1}(t^\prime) \nabla P_{>N^\gamma} w_{\leq n-1} \nabla F_n^\omega ~ \dtp \Big \|_{\YN} \\
&\lt N^{(1-\gamma)(\nu-1)+1-\sigma^\prime} \| \theta_{F,w;\leq n-1} \langle \nabla \rangle^{\nu} w_{\leq n-1} \|_{\XlN} \| \bra^{\sigma^\prime} F_n^\omega \|_{\SNprime} \\
&\lt N^{(1-\gamma)(\nu-1)+1-\sigma^\prime} \| \bra^{\sigma^\prime} F_n^\omega \|_{\SNprime} ~. 
\end{align*}
\end{proof}
\begin{rem}
Because of the importance of the term \( \nabla P_{>N^\gamma} u_{n-1} \cdot \nabla F_n^\omega \), we  informally justify \eqref{nl:eq_inhomogeneous_contribution} and describe the motivation behind the estimate. \\
The first power of \( N \) comes from the contribution of \( \nabla P_{N^\gamma} F_{\leq n-1}^\omega \cdot \nabla F_{n}^\omega \). It is bounded by
\begin{equation*}
\| \langle  \nabla \rangle^{\nu} \duh \nabla P_{N^\gamma} F_{\leq n-1}^\omega \cdot \nabla F_n^\omega \|_{L_t^\infty L_x^2} \lesssim
T^{\frac{1}{2}} N^{\nu-1} \| \nabla P_{N^\gamma} F_{\leq n-1}^\omega \|_{L_t^2L_x^\infty}  \| \nabla F_n^\omega \|_{L_t^\infty L_x^2}~. 
\end{equation*}
Thus, the resulting power is 
\begin{equation}\label{lin:cond_rough_rough}
N^{\nu-s-\gamma(\sigma^\prime-1)} = \underbrace{N^{\nu-1}}_{\text{derivatives}}  \cdot \underbrace{N^{-\gamma(\sigma^\prime-1)}}_{\text{derivatives on }F_{\leq n-1}^\omega} ~\cdot \underbrace{N^{1-s}}_{\text{derivatives on }F_n^\omega}.
\end{equation}
The second power of \( N \) comes from the contribution  \( \nabla P_{N^\gamma} w_{\leq n-1} \cdot \nabla F_n^\omega \). It is bounded by 
\begin{equation*}
\| \langle  \nabla \rangle^{\nu} \duh \nabla P_{N^\gamma} w_{\leq n-1} \cdot \nabla F_n^\omega \|_{L_t^\infty L_x^2}\lesssim T^{\frac{1}{2}} N^{\nu-1} \| P_{N^\gamma} \nabla w_{\leq n-1} \|_{L_t^\infty L_x^2}   \| \nabla F_n^\omega \|_{L_t^2 L_x^\infty}
\end{equation*}
Thus, the resulting power is
\begin{equation}\label{lin:cond_smooth_rough}
N^{(1-\gamma) (\nu-1)+1-\sigma^\prime}= \underbrace{N^{\nu-1}}_{\text{derivatives}}  \cdot \underbrace{N^{-\gamma(\nu-1)}}_{\text{derivatives on }w_{\leq n-1}}\cdot \underbrace{N^{1-\sigma^\prime}}_{\text{derivatives on }F_n^\omega}~.
\end{equation}
This estimate may seem counterintuitive, since the term with the higher frequency is placed in \( L_t^2 L_x^\infty \). However, this our only option to capitalize on the randomness, which enters through the probabilistic Strichartz estimate \eqref{lin:eq_probabilistic_strichartz}. In fact, switching the roles of \( w_{\leq n-1} \) and \( F_n^\omega \) above would not allow us to go below the deterministic restriction \( s>2 \). \\
\end{rem}

\subsection{Control of the nonlinear component \( w_n \)}

\begin{proof}[Proof of Proposition \ref{nl:prop_wn}:]~\\
We begin by showing the a-priori estimate for \( w_n \), which forms the main part of the proof. Afterwards, we will use contraction mapping to prove the existence and uniqueness of \( w_n \). This step could potentially be replaced by a soft argument, since all involved functions are smooth (with norms growing in \(N\)). 
\paragraph{A-priori bounds:}
We separate the proof into six cases, corresponding to the different terms in \eqref{nl:eq_truncated_wn}.\\

\emph{Case 1: Contribution of \( |\nabla F_n^\omega|^2 \).} Using \eqref{nl:eq_low_high_FF}, \eqref{nl:eq_high_low_FF}, and \eqref{nl:eq_high_high_FF}, we have that
\begin{equation*}  
\Big \|  \duh  \theta_{F;n}(t^\prime) \nabla F_n^\omega \cdot \nabla F_n^\omega  \dtp \Big\|_{\YN} 
\lt N^{\nu-s+1-\sigma^\prime} \| \bra^s F_n^\omega \|_{\XNprime} ~. 
\end{equation*}

\emph{Case 2:  Contribution of \( \nabla F_n^\omega \nabla w_n \).} Using \eqref{nl:eq_low_high_Fw}, \eqref{nl:eq_high_low_Fw}, and \eqref{nl:eq_high_high_Fw}, we have that
\begin{equation*}  
\Big \|  \duh  \theta_{F;n}(t^\prime) \nabla F_n^\omega \cdot \nabla w_n \dtp \Big \|_{\YN} 
\lesssim T^{\frac{1}{2}}  \| w_n \|_{\YN}~. 
\end{equation*}

\emph{Case 3:  Contribution of \( |\nabla w_n|^2 \).} Using \eqref{nl:eq_low_high_ww}, \eqref{nl:eq_high_low_ww}, and \eqref{nl:eq_high_high_ww}, we have that
\begin{equation*}
\Big \|  \duh \theta_{w;n}(t^\prime) \nabla w_n \cdot \nabla w_n \dtp \Big \|_{\YN} 
\lesssim T^{\frac{1}{2}} 
 \| w_n \|_{\YN} ~. 
\end{equation*}

\emph{Case 4:  Contribution of \(  \nabla F_{\leq n-1}^\omega \nabla w_n\).} Using \eqref{nl:eq_low_high_Fw}, \eqref{nl:eq_high_low_Fw}, and \eqref{nl:eq_high_high_Fw}, we have that
\begin{equation*}  
\Big \|  \duh  \theta_{F,w;\leq n-1}(t^\prime) \nabla F_{\leq n-1}^\omega \cdot \nabla w_n \dtp \Big \|_{\YN} 
\lesssim T^{\frac{1}{2}}  
 \| w_n \|_{\YN}~.
\end{equation*} 

\emph{Case 5: Contribution of \(  \nabla w_{\leq n-1} \nabla w_n\).} Using \eqref{nl:eq_low_high_ww}, \eqref{nl:eq_high_low_ww}, and \eqref{nl:eq_high_high_ww}, we have that
\begin{equation*}
\Big \| \duh \theta_{F,w;\leq n-1}(t^\prime) \nabla w_{\leq n-1}\cdot \nabla w_n \dtp \Big \|_{\YN} 
\lesssim T^{\frac{1}{2}} 
 \| w_n \|_{\YN} ~. 
\end{equation*}

\emph{Case 6: Contribution of \( \nabla P_{>N^\gamma} \nabla u_{n-1} \nabla F_n^\omega\).} This term was already estimated in Corollary \ref{nl:cor_inhomogeneous_contribution}. We have that
\begin{align*}
&\Big \| \duh \theta_{F,w;\leq n-1}(t^\prime) \nabla P_{>N^\gamma} u_{n-1} \nabla F_n^\omega \dtp \Big \|_{\YN} \\
&\lt \left( N^{\nu-s+\gamma(1-\sigma^\prime)} + N^{(1-\gamma) (\nu-1)+1-\sigma^\prime} \right) \left( \| \bra^s F_n^\omega \|_{\XNprime} + \| \bra^{\sigma^\prime} F_n^\omega \|_{\SNprime} \right)~. 
\end{align*}

Combining the estimates above, we obtain that 
\begin{equation*}
\| w_n \|_{\YN} \lesssim T^{\frac{1}{2}} \left( N^{\nu-s+\gamma(1-\sigma^\prime)} + N^{(1-\gamma) (\nu-1)+1-\sigma^\prime} \right) \left( \| \bra^s F_n^\omega \|_{\XNprime} + \| \bra^{\sigma^\prime} F_n^\omega \|_{\SNprime} \right) + T^{\frac{1}{2}} \| w_n \|_{\YN}~. 
\end{equation*}
Then, the a-priori bound follows by choosing \( T_0 >0 \) sufficiently small. 
\paragraph{Contraction Mapping:} ~\\
Due to the truncations using \( \theta\), we may work on the whole space \( \YN \). We set
\begin{equation}\label{nl:eq_gamma}
\begin{aligned}
\Gamma w(t)&:=
\duh  \theta_{F;n}(t^\prime) |\nabla F_n^\omega|^2\dtp  
		+2 \duh \theta_{F;n}(t^\prime) \nabla F_n^\omega \nabla w \dtp \\
		&~~+ \duh \theta_{w}(t^\prime) |\nabla w|^2\dtp
		+ 2 \duh \theta_{F,w;\leq n-1}(t^\prime) \nabla F_{\leq n-1}^\omega \nabla w \dtp \\
		&~~+ \duh \theta_{F,w;\leq n-1}(t^\prime) \nabla w_{\leq n-1} \nabla w \dtp \\
		&~~ + \duh \theta_{F,w;\leq n-1}(t^\prime) \nabla P_{>N^\gamma}  u_{n-1} \nabla F_n^\omega \dtp ~. 
\end{aligned}
\end{equation}
Here, the cutoff \( \theta_w(s) \) is defined by replacing \( w_n \) in the definition of \( \theta_{w;n}(s) \) with \( w \), see \eqref{it:eq_theta_wn}. The same arguments that led to the a-priori bound show that 
\begin{equation*}
\| \Gamma w \|_{\YN} \lesssim  T^{\frac{1}{2}} \left( N^{\nu-s-\gamma(\sigma^\prime-1)}+ N^{(1-\gamma)(\nu-1)+1-\sigma^\prime} \right) \left( \| \langle \nabla \rangle^s F_n^\omega \|_{\XNprime}+ \| \langle \nabla \rangle^{\sigma^\prime} F_n^\omega \|_{\SNprime}\right) + T^{\frac{1}{2}} \| w \|_{\YN}~. 
\end{equation*}
In particular, \( \Gamma \) maps \( \YN \) into \( \YN \). Thus, it suffices to prove for all \( v,w\in \YN \) that 
\begin{equation*}
\| \Gamma v - \Gamma w \|_{\YN} \lesssim T^{\frac{1}{2}} \| v - w \|_{\YN}~. 
\end{equation*}
For the linear terms in \( v \) and \( w \), this follows from the estimates above. Thus, it remains to control the quadratic term \( \theta_{v} |\nabla v|^2 - \theta_{w} |\nabla w|^2 \). We use a similar method as in the proof of \cite[Proposition 3.1]{BD99}. We define
\begin{equation*}
t_v := \sup\{ 0 \leq t \leq T\colon \| \langle \nabla \rangle^\nu v \|_{(\XNeta \bigcap \XlN)([0,t])} + \| \langle \nabla \rangle^{\sigma} v \|_{(\SNeta \bigcap \SlN)([0,t])}\leq 2 \}~. 
\end{equation*}

The time \( t_w \) is defined analogously.  Due to the continuity statement \eqref{prelim:eq_continuity}, we have that 
\begin{equation*}
\| 1_{[0,t_v]} ~ \langle \nabla \rangle^\nu v \|_{(\XNeta \bigcap \XlN)([0,T])}+ \quad \| 1_{[0,t_v]} ~ \langle \nabla \rangle^\sigma v \|_{(\SNeta \bigcap \SlN)([0,T])} \leq 2~.
\end{equation*}
To avoid confusion, we point out that the continuity statement \eqref{prelim:eq_continuity} is not enforced solely by the \( \XNT \)-norm, but comes from the definition of the space in \eqref{prelim:eq_spaces}. \\
Without loss of generality, we assume that \( t_v \leq t_w \). Using \eqref{nl:eq_low_high_ww}, \eqref{nl:eq_high_low_ww}, and \eqref{nl:eq_high_high_ww}, we have that 
\begin{align*}
&\Big\|  \duh \big( \theta_{v}(t^\prime) \ |\nabla v|^2-  \theta_{w}(t^\prime) |\nabla w|^2 \big) \dtp\Big\|_{\YN}\\
&\leq \Big \| \duh 1_{[0,t_v]}(t^\prime) (\theta_v(t^\prime)-\theta_w(t^\prime)) |\nabla v|^2 \dtp\Big \|_{\YN}  \\
&+ \Big\| \duh 1_{[0,t_v]}(t^\prime) \theta_w(t^\prime) \left( |\nabla v|^2 - |\nabla w|^2 \right) \dtp\Big \|_{\YN} \\
&+ \Big\| \duh 1_{(t_v,t_w]}(t^\prime) (\theta_v(t^\prime)-\theta_w(t^\prime)) |\nabla w|^2 \dtp \Big\|_{\YN} \\
&\lesssim T^{\frac{1}{2}} \| \theta_v - \theta_w \|_{L_t^\infty} \left( \| 1_{[0,t_v]} \bra^\nu v \|_{\XNeta \bigcap \XlN } + \| 1_{[0,t_v]} \bra^\sigma v \|_{\SNeta \bigcap \SlN} \right)^2\\
 &+ T^{\frac{1}{2}} \bigg( \| 1_{[0,t_v]} \bra^\nu v \|_{\XNeta \bigcap \XlN } + \| 1_{[0,t_v]} \bra^\sigma v \|_{\SNeta \bigcap \SlN}+ \| 1_{[0,t_v]} \bra^\nu w \|_{\XNeta \bigcap \XlN } \\
 &~~~+ \| 1_{[0,t_v]} \bra^\sigma w \|_{\SNeta \bigcap \SlN} \bigg) \| v -w \|_{\YN} \\
&+ T^{\frac{1}{2}} \| \theta_v - \theta_w \|_{L_t^\infty} \left( \| 1_{(t_v,t_w]} \bra^\nu w \|_{\XNeta \bigcap \XlN } + \| 1_{(t_v,t_w]} \bra^\sigma w \|_{\SNeta \bigcap \SlN} \right)^2\\
&\lesssim T^{\frac{1}{2}} \| v -w \|_{\YN}~. 
\end{align*}
Hence, \( \Gamma \) is a contraction on \( {\YN} \), and \( w_n \) can be defined as the unique fixed point of \( \Gamma \). 
\end{proof}

\section{Proof of the main theorem}\label{sec:final}
As in Section \ref{sec:lin}, any question regarding the (strong) measurability of the solutions is addressed in the appendix.
Before we begin with the proof of the main theorem, we collect all conditions on the parameters.
\paragraph{Parameter conditions:} First, we have the basic conditions
\begin{equation}\label{fin:eq_basic_conditions}
\nu > 2 > s > 1~, \quad \sigma= \nu - 1 -, \quad \sigma^\prime > \sigma, \quad \text{and} ~ \quad \gamma \in (0,1)~. 
\end{equation}
In order to use Proposition \ref{lin:prop_probabilistic_strichartz}, Proposition \ref{nl:prop_wn}, and Corollary \ref{nl:cor_inhomogeneous_contribution}, we require the major conditions
\begin{equation}\label{fin:eq_major_conditions}
\begin{aligned}
\sigma^\prime -s + 1 - \gamma (\sigma-1) - \frac{1}{2} (1-\gamma)&<0~,\\
\nu - s - \gamma (\sigma^\prime-1) &<0 ~, \\
(1-\gamma) (\nu-1) + 1- \sigma^\prime &<0~. 
\end{aligned}
\end{equation}
Because of \eqref{lin:eq_consistency_condition_strichartz} and \eqref{lin:eq_probabilistic_strichartz}, we also require the minor conditions
\begin{equation}\label{fin:eq_minor_conditions}
\nu < \frac{5}{2} \quad \text{and} \quad s> \sigma^\prime~. 
\end{equation}
In particular, if \eqref{fin:eq_basic_conditions}, \eqref{fin:eq_major_conditions}, and \eqref{fin:eq_minor_conditions} are satisfied, we can find an \( \eta \) that satisfies the conditions of Proposition \ref{nl:prop_wn}. \\

To complete the proof of the main theorem, we now have to prove the convergence of the iterates \( u_n \), remove the truncation in \eqref{it:eq_truncated_Fn} and \eqref{it:eq_truncated_wn} by choosing a small random time \( T(\omega)>0 \), and optimize the parameters. 

\begin{proof}[\textbf{Proof of the main theorem:}]~\\
Let \( F_n^\omega\), \( w_n\), and \( u_n \) be as in \eqref{it:eq_truncated_Fn}, \eqref{it:eq_truncated_wn}, and \eqref{it:eq_truncated_un}. As before, we have eliminated the subcript \( \theta \) from our notation.
First, we show the convergence of the iterates \( u_n \). Assuming that the parameters satisfy  \eqref{fin:eq_basic_conditions}, \eqref{fin:eq_major_conditions}, and \eqref{fin:eq_minor_conditions}, we prove that there exists a random function \( {u\colon \Omega \times [0,T_0] \times \rthree \rightarrow \mathbb{R} }\) s.t.
\begin{equation}\label{fin:eq_convergence}
\begin{aligned}
u_n \rightarrow u \quad &\text{in} \quad L_\omega ^2 C_t^0 H_x^s(\Omega \times [0,T_0]\times \rthree) ~ \text{and} ~L_\omega^2 L_t^2 W_x^{\sigma,\infty}(\Omega \times [0,T_0]\times \rthree)~,\\
\partial_t u_n \rightarrow \partial_t u \quad &\text{in} \quad  L_\omega ^2 C_t^0 H_x^{s-1}(\Omega \times [0,T_0]\times \rthree) ~. 
\end{aligned}
\end{equation}
Here, \( T_0 > 0 \) is as in Proposition \ref{nl:prop_wn}.

Let \( \epsilon > 0 \) be sufficiently small depending on the parameters above. We show the convergence of the series \( \sum_{m=0}^{\infty} F_m^\omega \) and \( \sum_{m=0}^\infty w_m \) in  \( L_\omega ^2C_t^0 H_x^s \) and \( L_\omega^2 L_t^2 W_x^{\sigma,\infty} \). The convergence of the time-derivatives follows from a similar argument. \\
Let \( 0 \leq n_- < n_+ < \infty \) be arbitrary. Writing \(M=2^m\), we obtain from Minkowski's integral inequality and the definition of \( \XMprime \) that
\begin{align*}
&\Big \| \sum_{m=n_-}^{n_+} \langle \nabla \rangle^s F_m^\omega \Big \|_{L_\omega^2 L_t^\infty L_x^2} 
\lesssim \Big \| \sum_{m=n_-}^{n_+} \langle \nabla \rangle^s P_N F_m^\omega \Big \|_{L_\omega^2 L_t^\infty \ell_N^2 L_x^2}
\lesssim  \Big \| \sum_{m=n_-}^{n_+} \langle \nabla \rangle^s P_N F_m^\omega \Big \|_{L_\omega^2 \ell_N^2 L_t^\infty  L_x^2}\\
&\lesssim \Big \| \sum_{m=n_-}^{n_+} \| \langle \nabla \rangle^s P_N F_m^\omega \|_{L_t^\infty L_x^2} \Big \|_{L_\omega^2 \ell_N^2} 
\lesssim \Big \| \sum_{m=n_-}^{n_+} \max\left( \frac{N}{M}, \frac{M}{N} \right)^{-D^\prime} \| \langle \nabla \rangle^s  F_m^\omega \|_{\XMprime} \Big \|_{L_\omega^2 \ell_N^2}~.
\end{align*}
By using Corollary  \ref{lin:cor_frequency_Fnomega}, it follows that 
\begin{align*}
&\Big \| \sum_{m=n_-}^{n_+} \max\left( \frac{N}{M}, \frac{M}{N} \right)^{-D^\prime} \| \langle \nabla \rangle^s  F_m^\omega \|_{\XMprime} \Big \|_{L_\omega^2 \ell_N^2}\\
&\lesssim  \Big \| \sum_{m=n_-}^{n_+} \max\left( \frac{N}{M}, \frac{M}{N} \right)^{-D^\prime} \| ( \widetilde{P}_M f_0^\omega , \widetilde{P}_M f_1^\omega) \|_{H_x^s\times H_x^{s-1}} \Big \|_{L_\omega^2 \ell_N^2}\\ 
&\lesssim \Big \| \left ( \sum_{m=n_-}^{n_+} \| (\widetilde{P}_M f_0^\omega, \widetilde{P}_M f_1^\omega) \|_{H_x^s\times H_x^{s-1}}^2 \right)^{\frac{1}{2}} \Big \|_{L_\omega^2} 
\lesssim \left(  \sum_{m=n_-}^{n_+} \| (\widetilde{P}_M f_0, \widetilde{P}_M f_1) \|_{H_x^s\times H_x^{s-1}}^2 \right)^{\frac{1}{2}} ~. 
\end{align*}
This proves that the series \( \sum_{m=0}^\infty F_m^\omega \) is Cauchy in \( L_\omega^2 C_t^0 H_x^s \). From Proposition \ref{lin:prop_probabilistic_strichartz}, we have that
\begin{equation}\label{fin:eq_Fm_Strichartz}
\| \langle \nabla \rangle^{\sigma} F_m^\omega \|_{L_\omega^2 L_t^2 L_x^\infty} \lesssim \| \langle \nabla \rangle^{\sigma^\prime} F_m^\omega \|_{L_\omega^2 \SMprime} \lesssim M^{-\epsilon}  \| (\widetilde{P}_M f_0, \widetilde{P}_M f_1) \|_{H_x^s\times H_x^{s-1}}~. 
\end{equation}
This proves the convergence of \( \sum_{m=0}^\infty F_m^\omega \) in \( L_\omega^2 L_t^2 W_x^{\sigma,\infty} \). \\
From Proposition \ref{nl:prop_wn}, we have that 
\begin{equation*}
\| \langle \nabla \rangle^{\nu} w_m \|_{L_t^\infty L_x^2} + \|\langle \nabla \rangle^{\sigma} w_m \|_{L_t^2 L_x^\infty} \lesssim
T^{\frac{1}{2}} M^{-\epsilon} \left( \| \langle \nabla \rangle^{s} F_m^\omega \|_{\XMprime} + \| \langle \nabla \rangle^{\sigma^\prime} F_m^\omega \|_{\SMprime} \right)~. 
\end{equation*}
After taking moments in \( \omega \), the convergence then follows from Corollary \ref{lin:cor_frequency_Fnomega}  and Proposition \ref{lin:prop_probabilistic_strichartz}. \\

Second, we show that there exist random times \( T(\omega) \) s.t. \eqref{intro:eq_u_duhamel} holds. To eliminate the cutoff, it suffices to choose \( T(\omega)>0 \) s.t.
\begin{equation*}
\sum_{m=0}^{\infty} \Big( \| \langle \nabla \rangle^{\sigma^\prime} F_{m}^\omega \|_{\SMprime([0,T(\omega)])} + \| w_m \|_{\YM([0,T(\omega)])} \Big)
\leq 1~. 
\end{equation*}
Using the continuity statement \eqref{prelim:eq_continuity} and the estimate \eqref{fin:eq_Fm_Strichartz}, we have for a.e. \( \omega \in \Omega \) that
\begin{equation*}
t \in [0,T_0] \mapsto \sum_{m=0}^\infty \| \langle \nabla \rangle^{\sigma^\prime} F_m^\omega \|_{\SMprime([0,t])} 
\end{equation*}
is continuous and equals zero at \( t= 0 \). As a consequence, the random time
\begin{equation*}
T_1(\omega) := \sup\bigg\{ 0 \leq t \leq T_0 \colon  \sum_{m=0}^\infty \| \langle \nabla \rangle^{\sigma^\prime} F_m^\omega \|_{\SMprime([0,t])}\leq \frac{1}{2} \bigg\} 
\end{equation*}
is almost surely positive. To control the nonlinear components \( w_m \), we recall from Proposition \ref{nl:prop_wn} that
\begin{equation*}
\| w_m \|_{\YM([0,T])}\lesssim T^{\frac{1}{2}} M^{-\epsilon} \left( \| \langle \nabla \rangle^s F_m^\omega \|_{\XMprime([0,T])}+ \| \langle \nabla \rangle^{\sigma^\prime} F_m^\omega \|_{\SMprime([0,T])}\right)
\end{equation*}
Using Corollary \ref{lin:cor_frequency_Fnomega} and Proposition \ref{lin:prop_probabilistic_strichartz}, we have almost surely that
\begin{equation*}
\sum_{m=0}^\infty M^{-\epsilon}  \left( \| \langle \nabla \rangle^s F_m^\omega \|_{\XMprime([0,T_0])}+ \| \langle \nabla \rangle^{\sigma^\prime} F_m^\omega \|_{\SMprime([0,T_0])}\right) < \infty~.
\end{equation*}
Thus, the random time 
\begin{equation*}
T_2(\omega) := \sup \bigg \{ 0\leq t \leq T_0\colon \sum_{m=0}^\infty \| w_m \|_{\YM} \leq \frac{1}{2} \bigg\}
\end{equation*}
is almost surely positive. Setting \( T(\omega) = \min( T_1(\omega), T_2(\omega) ) \), we obtain for a.e. \( w\in \Omega \) that 
\begin{equation*}
u_n(t)= W(t) (f_0^\omega, f_1^\omega) + \duh |\nabla u_n(t^\prime)|^2 \dtp \qquad  \forall n\geq 0 ~ \text{and} ~\forall t \in [0,T(\omega)]   ~ . 
\end{equation*} 
Then, \eqref{intro:eq_u_duhamel} follows from the convergence of the iterates \( u_n \). \\

Third, we have to determine nearly optimal parameters \( (s,\nu,\sigma^\prime,\gamma) \). We discretized the parameter \( \gamma \in (0,1) \) and used a linear programming solver to find the remaining parameters  \( (s,\nu,\sigma^\prime ) \) with an almost optimal value of \( s \). This leads to 
\begin{equation}\label{fin:eq_parameters}
(s,~ \nu,~ \sigma^\prime,~ \gamma ) = (1.9840,~ 2.1001,~ 1.13205,~ 0.88 )~.
\end{equation}
\end{proof}\vspace{-4ex}

\section{Proof of the bilinear estimates}\label{section:proof_bilinear}

In this section, we present the proofs of Lemma \ref{nl:lem_low_high}, Lemma \ref{nl:lem_high_low}, and Lemma \ref{nl:lem_high_high}.

\begin{proof}[Proof of Lemma \ref{nl:lem_low_high}]~\\
First, we prove the estimates \eqref{nl:eq_low_high_FF} and \eqref{nl:eq_low_high_PFF}. Let \( H \in \{ P_{>N^\gamma} G, G\} \). Then, we for all \( M \geq 1 \) that
\begin{align*}
&\| \bra^\nu P_M \Plh(H,F) \|_{L_t^\infty L_x^2} + \| \bra^{\nu-1} \partial_t P_M \Plh(H,F) \|_{L_t^\infty L_x^2}+ \| \bra^\sigma P_M \Plh(H,F) \|_{L_t^2L_x^\infty} \\
&\lesssim M^{\nu-1} \sum_{1\leq L \ll M} \sum_{K\sim M} \| \nabla P_L H \cdot \nabla P_K F \|_{L_t^1 L_x^2} \\
&\lt M^{\nu-s} \sum_{1\leq L \ll M} L^{1-\sigma^\prime} \| \langle \nabla \rangle^{\sigma^\prime} P_L H \|_{L_t^2 L_x^\infty} \sum_{K\sim M} \| \bra^s P_K F \|_{L_t^\infty L_x^2} ~.
\end{align*}
After multiplying by \( c_{N,D}( M) \) and summing in \( M\), we obtain for all \( D^\prime > 2D \) and \( D>\eta \) that 
\begin{align*}
& \|\Plh(H,F) \|_{\YN} \\
&\lt \sum_{M\geq 1} \sum_{1\leq L \ll M} M^{\nu-s} L^{1-\sigma^\prime} \max \left( \frac{N}{M} , \frac{M}{N} \right)^{-D}  \| \bra^{\sigma^\prime} P_L H \|_{L_t^2 L_x^\infty} \| \bra^s F \|_{\XNprime}~.
\end{align*}

We now distinguish the two different possibilities for \( H \). If \( H= G \), then 
\begin{align*}
&\sum_{M\geq 1} \sum_{1\leq L \ll M} M^{\nu-s} L^{1-\sigma^\prime} \max \left( \frac{N}{M} , \frac{M}{N} \right)^{-D}  \| \bra^{\sigma^\prime} P_L H \|_{L_t^2 L_x^\infty}\\
&\lesssim \sum_{M\geq 1} \sum_{1\leq L \ll M} M^{\nu-s} L^{1-\sigma^\prime} \max \left( \frac{N}{M} , \frac{M}{N} \right)^{-D}  \max \left( \frac{N}{L} , \frac{L}{N} \right)^{-D^\prime} \| \bra^{\sigma^\prime} G \|_{\SNprime} \\
&\lesssim N^{\nu-s+1-\sigma^\prime}  \| \bra^{\sigma^\prime} G \|_{\SNprime} ~. 
\end{align*}
If \( H= P_{>N^\gamma} G \), then 
\begin{align*}
&\sum_{M\geq 1} \sum_{1\leq L \ll M} M^{\nu-s} L^{1-\sigma^\prime} \max \left( \frac{N}{M} , \frac{M}{N} \right)^{-D}  \| \bra^{\sigma^\prime} P_L H \|_{L_t^2 L_x^\infty}\\
&\lesssim \sum_{M\geq N^\gamma} \sum_{N^\gamma\leq L \ll M} M^{\nu-s} L^{1-\sigma^\prime} \max \left( \frac{N}{M} , \frac{M}{N} \right)^{-D}  ~ \| \bra^{\sigma^\prime}  G \|_{\SlNprime}\\
&\lesssim N^{\nu-s+\gamma(1-\sigma^\prime)} \| \bra^{\sigma^\prime}  G \|_{\SlNprime}~. 
\end{align*}
Second, we prove \eqref{nl:eq_low_high_PwF}. For any \( M \geq 1 \), we have that 
\begin{align*}
&~\| \bra^\nu P_M \Plh(P_{>N^\gamma }v,F) \|_{L_t^\infty L_x^2} +\| \bra^{\nu-1} \partial_t P_M \Plh(P_{>N^\gamma }v,F) \|_{L_t^\infty L_x^2} \\
&+ \| \bra^\sigma P_M \Plh(P_{>N^\gamma} v,F) \|_{L_t^2L_x^\infty} \\
&\lesssim M^{\nu-1} \sum_{N^\gamma\leq L \ll M} \sum_{K\sim M} \| \nabla P_L v \cdot \nabla P_K F \|_{L_t^1 L_x^2}\\
&\lt M^{\nu-\sigma^\prime} \sum_{N^\gamma \leq L \ll M} L^{1-\nu} \| \bra^\nu P_L v \|_{L_t^\infty L_x^2} ~ \sum_{K\sim M} \| \bra^{\sigma^\prime} P_K F \|_{L_t^2 L_x^\infty}~.
\end{align*}
After multiplying with \( c_{N,D}(M) \) and summing in \( M \), we obtain for all \( D^\prime> 2D \) and \( D>\eta \) that 
\begin{align*}
&\|  \Plh(P_{>N^\gamma }v,F) \|_{\YN} \\
&\lt \sum_{M\geq 1} \sum_{N^\gamma \leq L \ll M} M^{\nu-\sigma^\prime} L^{1-\nu} \max \left( \frac{N}{M} , \frac{M}{N} \right)^{-D} \| \bra^\nu v \|_{\XlN} \| \bra^{\sigma^\prime} F \|_{\XNprime} \\
&\lt N^{(1-\gamma)(\nu-1)+1-\sigma^\prime}\| \bra^\nu v \|_{\XlN} \| \bra^{\sigma^\prime} F \|_{\XNprime}~.
\end{align*}
Third, we prove \eqref{nl:eq_low_high_Fw} and \eqref{nl:eq_low_high_ww}. For any \( M \geq 1 \), we have that 
\begin{align*}
&\| \bra^\nu  \Plh(G,w) \|_{L_t^\infty L_x^2} +\| \bra^{\nu-1} \partial_t  \Plh(G,w) \|_{L_t^\infty L_x^2}+ \| \bra^\sigma \Plh(G,w) \|_{L_t^2 L_x^\infty} \\
&\lesssim M^{\nu-1} \sum_{1\leq L \ll M} \sum_{K\sim M} \| \nabla P_L G \cdot \nabla  P_K w \|_{L_t^1L_x^2}\\
&\lt  \sum_{L\geq 1} L^{1-\sigma^\prime} \| \bra^{\sigma^\prime} P_L G \|_{L_t^2 L_x^\infty} \sum_{K\sim M} \| \bra^\nu P_K w \|_{L_t^\infty L_x^2} \\
&\lt \| \bra^{\sigma^\prime} G \|_{\SNprime}  \sum_{K\sim M} \| \bra^\nu P_K w \|_{L_t^\infty L_x^2} ~. 
\end{align*}
After multiplying by \( c_{N,\eta}(M)+c_{\leq N,D}(M) \) and summing in \( M \geq 1 \), we obtain \eqref{nl:eq_low_high_Fw}. The estimate \eqref{nl:eq_low_high_ww} follows from exactly the same argument. \\
This finishes the proof of the low-high bilinear estimates.

\end{proof}

\begin{proof}[Proof of Lemma \ref{nl:lem_high_low}]~\\
First, we prove \eqref{nl:eq_high_low_FF}. For any \( M \geq 1 \), we have for all sufficiently large \( D^\prime >0 \) that
\begin{align*}
&\| \langle \nabla \rangle^{\nu} P_M \Phl(G,F) \|_{L_t^\infty L_x^2}+\| \langle \nabla \rangle^{\nu-1} \partial_t P_M \Phl(G,F) \|_{L_t^\infty L_x^2}  + \| \langle \nabla \rangle^{\sigma} P_M \Phl(G,F) \|_{L_t^2 L_x^\infty}\\
&\lesssim M^{\nu-1} \sum_{L\sim M} \sum_{K\ll M} \| \nabla P_L G \cdot \nabla P_K F \|_{L_t^1 L_x^2} \\
&\lt M^{\nu-\sigma^\prime} \max \left( 1, \frac{M}{N} \right)^{-D^\prime} \sum_{K \ll M} K^{1-s} \max \left( \frac{N}{K}, \frac{K}{N} \right)^{-D^\prime} \| \bra^{\sigma^\prime} G \|_{\SlNprime} \| \bra^s F \|_{\XNprime} \\
&\lt M^{\nu-\sigma^\prime} \Big( \Is{M\lesssim N} M^{1-s} \left( \frac{N}{M} \right)^{-D^\prime} + \Is{M\gg N} N^{1-s} \left( \frac{M}{N} \right)^{-D^\prime} \Big)\| \bra^{\sigma^\prime} G \|_{\SlNprime} \| \bra^s F \|_{\XNprime}\\
&\lt N^{\nu-s+1-\sigma^\prime} \max \left( \frac{M}{N}, \frac{N}{M} \right)^{-2D} \| \bra^{\sigma^\prime} G \|_{\SlNprime} \| \bra^s F \|_{\XNprime}~. 
\end{align*}
After multiplying by \( c_{N,D}(M) \) and summing in \( M \geq 1 \), this yields an acceptable contribution. \\
Second, we prove \eqref{nl:eq_high_low_wF}. For any \( M \geq 1 \), we have that
\begin{align*}
&\| \langle \nabla \rangle^{\nu} P_M \Phl(v,F) \|_{L_t^\infty L_x^2}+\| \langle \nabla \rangle^{\nu-1}  P_M \partial_t \Phl(v,F) \|_{L_t^\infty L_x^2}  + \| \langle \nabla \rangle^{\sigma} P_M \Phl(v,F) \|_{L_t^2 L_x^\infty}\\
&\lesssim M^{\nu-1} \sum_{L\sim M} \sum_{K\ll M} \| \nabla P_L v \cdot \nabla P_K F \|_{L_t^1 L_x^2} \\
&\lt  \Big( \sum_{K\ll M} K^{1-\sigma^\prime} \max \left( \frac{N}{K} , \frac{K}{N} \right)^{-D^\prime} \Big) \Big(\sum_{L\sim M} \| \bra^\nu P_L v \|_{L_t^\infty L_x^2} \Big) \| \bra^{\sigma^\prime} F \|_{\SNprime} \\
&\lt \Big( \Is{M\lesssim N} M^{1-\sigma^\prime} \Big( \frac{N}{M} \Big)^{-D^\prime} + \Is{M\gg N} N^{1-\sigma^\prime} \Big)\Big(\sum_{L\sim M} \| \bra^\nu P_L v \|_{L_t^\infty L_x^2} \Big) \| \bra^{\sigma^\prime} F \|_{\SNprime} ~. 
\end{align*}
After multiplying with \( c_{N,D}(M) \) and summing in \( M \geq 1 \), it follows that 
\begin{align*}
&\|  \Phl(v,F) \|_{\YN} \\
&\lt \Big( \sum_{1\leq M\lesssim N} M^{1-\sigma^\prime} \Big( \frac{N}{M} \Big)^{D-D^\prime} \Big)  \| \bra^\nu  v \|_{\XlN} \| \bra^{\sigma^\prime} F \|_{\SNprime}\\
 &+ T^{\frac{1}{2}}N^{1-\sigma^\prime} \sum_{M\gg N} \sum_{L\sim M} c_{N,D}(M)  \| \bra^\nu P_L v \|_{L_t^\infty L_x^2}  \| \bra^{\sigma^\prime} F \|_{\SNprime} \\
 &\lt  N^{1-\sigma^\prime} \| \bra^\nu  v \|_{\XlN} \| \bra^{\sigma^\prime} F \|_{\SNprime}
\end{align*}
Third, we prove \eqref{nl:eq_high_low_Fw}. For any \( M \geq 1 \), it follows from \( \eta < \nu-1 \) that
\begin{align*}
&\| \langle \nabla \rangle^{\nu} P_M \Phl(G,w) \|_{L_t^\infty L_x^2} +\| \langle \nabla \rangle^{\nu-1}  P_M \partial_t \Phl(G,w) \|_{L_t^\infty L_x^2} + \| \langle \nabla \rangle^{\sigma} P_M \Phl(G,w) \|_{L_t^2 L_x^\infty}\\
&\lesssim M^{\nu-1} \sum_{L\sim M} \sum_{K\ll M} \| \nabla P_L G \cdot \nabla P_K w \|_{L_t^1 L_x^2} \\
&\lt M^{\nu-\sigma^\prime} \max\left( 1, \frac{M}{N} \right)^{-D^\prime}  \Big( \sum_{K\ll M} K^{1-\nu} \max\left( \frac{N}{K}, \frac{K}{N} \right)^{-\eta} \Big) \| \bra^{\sigma^\prime} G \|_{\SlNprime} \| \bra^\nu w \|_{\XNeta} \\
&\lt M^{\nu-\sigma^\prime} \max\left( 1, \frac{M}{N} \right)^{-D^\prime} N^{-\eta} \| \bra^{\sigma^\prime} G \|_{\SlNprime} \| \bra^\nu w \|_{\XNeta}~. 
\end{align*}
After multiplying with \( c_{N,\eta}(M) + c_{\leq N,\eta}(M) \) and summing in \( M \geq 1 \), the total contribution is bounded by 
\begin{align*}
&T^{\frac{1}{2}} \Big( \sum_{1\leq M\lesssim N} M^{\nu-\sigma^\prime-\eta} + \sum_{M\gg N} M^{\nu-\sigma^\prime} N^{-\eta} \Big( \frac{M}{N} \Big)^{D-D^\prime} \Big) \| \bra^{\sigma^\prime} G \|_{\SlNprime} \| \bra^\nu w \|_{\XNeta} \\
&\lt  \| \bra^{\sigma^\prime} G \|_{\SlNprime} \| \bra^\nu w \|_{\XNeta}~,
\end{align*}
where we used that \( \eta > \nu - \sigma^\prime \). \\
Finally, we prove \eqref{nl:eq_high_low_ww}, where we argue as in the proof of \eqref{nl:eq_high_low_wF}. For any \( M \geq 1 \), it holds that
\begin{align*}
&\| \langle \nabla \rangle^{\nu} P_M \Phl(v,w) \|_{L_t^\infty L_x^2}+\| \langle \nabla \rangle^{\nu-1} P_M \partial_t\Phl(v,w) \|_{L_t^\infty L_x^2} + \| \langle \nabla \rangle^{\sigma} P_M \Phl(v,w) \|_{L_t^2 L_x^\infty}\\
&\lt  \Big( \sum_{K\ll M} K^{1-\sigma} \max \left( \frac{N}{K} , \frac{K}{N} \right)^{-\eta} \Big) \Big(\sum_{L\sim M} \| \bra^\nu P_L v \|_{L_t^\infty L_x^2} \Big) \| \bra^{\sigma^\prime} w \|_{\SNeta} \\
&\lt \Big( \Is{M\lesssim N} M^{1-\sigma} \Big( \frac{N}{M} \Big)^{-\eta} + \Is{M\gg N} N^{1-\sigma} \Big)\Big(\sum_{L\sim M} \| \bra^\nu P_L v \|_{L_t^\infty L_x^2} \Big) \| \bra^{\sigma} w \|_{\SNeta} ~. 
\end{align*}
In the last line, we used that \( \eta > \sigma-1 \). After multiplying with \( c_{N,\eta}(M)+c_{\leq N,D}(M) \), the total contribution is bounded by
\begin{align*}
&\|  \Phl(v,w) \|_{\YN} \\
&\lt \Big( \sum_{1\leq M\lesssim N} M^{1-\sigma}  \Big)  \| \bra^\nu  v \|_{\XlN} \| \bra^{\sigma^\prime} w \|_{\SNeta}\\
 &+ T^{\frac{1}{2}}N^{1-\sigma} \sum_{M\gg N} \sum_{L\sim M} c_{N,D}(M)  \| \bra^\nu P_L v \|_{L_t^\infty L_x^2}  \| \bra^{\sigma^\prime} w \|_{\SNeta} \\
 & \lt \| \bra^\nu  v \|_{\XlN} \| \bra^{\sigma^\prime} w \|_{\SNeta}~. 
 \end{align*}
 This finishes the proof of the high-low bilinear estimates. 
\end{proof}

\begin{proof}[Proof of Lemma \ref{nl:lem_high_high}]
We begin with the proof of \eqref{nl:eq_high_high_FF}. For any \( M \geq 1 \), we have that 
\begin{align*}
&\| \langle \nabla \rangle^{\nu} P_M \Phh(G,F) \|_{L_t^\infty L_x^2}+\| \langle \nabla \rangle^{\nu-1} \partial_t P_M \Phh(G,F) \|_{L_t^\infty L_x^2}  + \| \langle \nabla \rangle^{\sigma} P_M \Phh(G,F) \|_{L_t^2 L_x^\infty}\\
&\lesssim M^{\nu-1} \sum_{L\sim K \gg M} \| \nabla P_L G\cdot \nabla P_K F \|_{L_t^1L_x^2}\\
&\lt M^{\nu-1} \left( \sum_{L\sim K\gg M} L^{1-\sigma^\prime} K^{1-s} \max \Big( 1,\frac{L}{N} \Big)^{-D^\prime} \max\Big( \frac{N}{K}, \frac{K}{N} \Big)^{-D^\prime} \right) \| \bra^{\sigma^\prime} G \|_{\SlNprime} 
\| \bra^s F \|_{\XNprime}\\
&\lt M^{\nu-1} \left(\Is{M\lesssim N} N^{2-\sigma^\prime-s} + \Is{M\gg N} M^{2-\sigma^\prime-s} \Big( \frac{M}{N} \Big)^{-2D^\prime} \right)\| \bra^{\sigma^\prime} G \|_{\SlNprime} 
\| \bra^s F \|_{\XNprime} \\
&= T^{\frac{1}{2}} N^{\nu-s+1-\sigma^\prime} \left( \Is{M\lesssim N} \Big( \frac{M}{N} \Big)^{\nu-1} + \Is{M\gg N} \Big( \frac{M}{N} \Big)^{-2D^\prime+\nu-s+1-\sigma^\prime} \right) \| \bra^{\sigma^\prime} G \|_{\SlNprime} 
\| \bra^s F \|_{\XNprime}~. 
\end{align*}
Since \( \eta < \nu-1 \), we may multiply by \( c_{N,\eta}(M) + c_{\leq N,D}(M) \) and sum in \( M \geq 1\). \\
Next, we proof \eqref{nl:eq_high_high_wF}. For any \( M \geq 1 \), we have that 
\begin{align*}
&\| \langle \nabla \rangle^{\nu} P_M \Phh(v,F) \|_{L_t^\infty L_x^2}+ \| \langle \nabla \rangle^{\nu-1} P_M \partial_t \Phh(v,F) \|_{L_t^\infty L_x^2}  + \| \langle \nabla \rangle^{\sigma} P_M \Phh(v,F) \|_{L_t^2 L_x^\infty}\\
&\lesssim M^{\nu-1} \sum_{L\sim K \gg M} \| \nabla P_L v\cdot \nabla P_K F \|_{L_t^1L_x^2}\\
&\lt M^{\nu-1} \left( \sum_{L\sim K\gg M} L^{1-\nu} K^{1-\sigma^\prime} \max \Big( 1,\frac{L}{N} \Big)^{-D} \max\Big( \frac{N}{K}, \frac{K}{N} \Big)^{-D^\prime} \right) \| \bra^{\nu} v \|_{\XlN} 
\| \bra^{\sigma^\prime} F \|_{\SNprime}\\
&\lt M^{\nu-1} \left( \Is{M\lesssim N} N^{2-\nu-\sigma^\prime} + \Is{M\gg N} M^{2-\nu-\sigma^\prime} \Big( \frac{M}{N} \Big)^{-D-D^\prime} \right) \| \bra^{\nu} v \|_{\XlN}  \| \bra^{\sigma^\prime} F \|_{\SNprime} \\
&\lt N^{1-\sigma^\prime} \left( \Is{M\lesssim N} \Big( \frac{M}{N} \Big)^{\nu-1} + \Is{M\gg N} \Big( \frac{M}{N} \Big)^{-D-D^\prime+1-\sigma} \right)  \| \bra^{\nu} v \|_{\XlN} \| \bra^{\sigma^\prime} F \|_{\SNprime}~. 
\end{align*}
Since \( \eta < \nu-1 \), we may multiply by \( c_{N,\eta}(M) + c_{\leq N,D}(M) \) and sum in \( M \geq 1\). \\
Finally, we prove \eqref{nl:eq_high_high_Fw} and \eqref{nl:eq_high_high_ww}. For any \( M \geq 1 \), we have that 
\begin{align*}
&\| \langle \nabla \rangle^{\nu} P_M \Phh(G,w) \|_{L_t^\infty L_x^2}+\| \langle \nabla \rangle^{\nu-1} P_M \partial_t \Phh(G,w) \|_{L_t^\infty L_x^2} + \| \langle \nabla \rangle^{\sigma} P_M \Phh(G,w) \|_{L_t^2 L_x^\infty}\\
&\lesssim M^{\nu-1} \sum_{L\sim K \gg M} \| \nabla P_L G\cdot \nabla P_K w \|_{L_t^1L_x^2}\\
&\lt M^{\nu-1} \sum_{L\sim K\gg M} L^{1-\sigma^\prime} K^{1-\nu} \max\Big( 1,\frac{L}{N} \Big)^{-D} \max\Big( \frac{N}{K},\frac{K}{N} \Big)^{-\eta} \| \bra^{\sigma^\prime} G \|_{\SlN} \| \bra^\nu v \|_{\XNeta} \\
&\lt M^{1-\sigma^\prime} \left( \Is{M\lesssim N} \Big( \frac{M}{N} \Big)^{\eta} + \Is{M\gg N} \Big ( \frac{M}{N} \Big)^{-D-\eta} \right) \| \bra^{\sigma^\prime} G \|_{\SlN} \| \bra^\nu v \|_{\XNeta}
\end{align*}
In the evaluation of the sum, we have used that \( \eta < \nu-1 < \nu + \sigma^\prime - 2 \). After multiplying by \( c_{N,\eta}(M)+ c_{\leq N,D}(M) \) and summing in \( M \geq 1 \), we see that 
\begin{equation*}
\| \Phh(G,w) \|_{\YN} \lt    \| \bra^{\sigma^\prime} G \|_{\SlNprime} \| \bra^\nu v \|_{\XNeta}~. 
\end{equation*}
Since we only used \( \sigma^\prime > 1 \), the same argument also yields \eqref{nl:eq_high_high_ww}. This finishes the proof of the high-high bilinear estimates. 
\end{proof}

\appendix
\section{Appendix}
\subsection*{Strong measurability of \( F_n^\omega \) and \( w_n \)}

In this section, we prove the strong measurability of the iterates. As before, let \( (\Omega, \mathscr{F},\mathbb{P} ) \) be the given probability space. We recall the following definition from the theory of Bochner-integration.
\begin{definition}
Let \( E \) be a Banach space. A function \( v\colon \Omega \rightarrow E \)  is called simple if there exist measurable sets \( F_i \in \mathscr{F} \) and vectors \( x_i \in E \), \( i=1,\hdots,k \), such that
\begin{equation*}
v= \sum_{i=1}^k 1_{F_i}(\omega) ~ x_i ~. 
\end{equation*}
A function \( v\colon \Omega \rightarrow E \)  is called strongly measurable (or strongly \( \mathscr{F}\)-measurable) if it can be written as the pointwise limit of simple functions. Finally, a function \( v\colon \Omega \rightarrow E \) is called strongly \( \mathbb{P}\)-measurable if there exists a strongly measurable function \( \widetilde{v} \colon \Omega \rightarrow E \) such that \( v (\omega)= \widetilde{v}(\omega) \) holds \( \mathbb{P}\)-almost surely. 
\end{definition}

The following two properties follow directly from the definition. 
\begin{lem}Let \( E,E_1,E_2, \) and \( F \) be Banach spaces. \label{appendix:lem_properties}
\begin{enumerate}[itemsep=0ex]
\item If \( v\colon \Omega \rightarrow E \)  is strongly measurable and \( \phi\colon E \rightarrow F \) is continuous (but possibly nonlinear),  then the composition \( \phi\circ v\colon \Omega \rightarrow F \) is strongly measurable.
\label{appendix:enum_cont} 
\item If \( v_i \colon \Omega \rightarrow E_i \), \( i=1,2 \), are strongly measurable, then \( (v_1,v_2) \colon \Omega \mapsto E_1 \times E_2 \) is strongly measurable. \label{appendix:enum_prod}
\end{enumerate}\vspace{-4ex}
\end{lem}

We are now ready to prove the main proposition of this section. Recall the definition of the sigma-algebra \( \mathscr{F}_n := \sigma( g_l \colon \| l\|_2 < 2^n) \), where \( n\in \mathbb{N}_0 \).
\begin{prop}\label{appendix:prop_measurability}
Let \( F_n^\omega\), \( F_{n,k} \), and \( w_n\) be as in \eqref{it:eq_truncated_Fn}, \eqref{it:eq_truncated_Fnk}, and \eqref{it:eq_truncated_wn}. Furthermore, let 0 \( \leq T \leq T_0 \) be as in Theorem \ref{main_thm}. Then, we have for all \( n \geq 0 \) that 
\begin{enumerate}[itemsep=0ex]
\item the functions \( \omega \mapsto \bra^s F_n^\omega \in \XNprime  \), \( \omega \mapsto \bra^{s-1} \partial_t F_n^\omega \in \XNprime \), and \( \omega \mapsto \bra^{\sigma^\prime} F_n^\omega \in \SNprime \) are strongly \( \mathscr{F}_n \)-measurable, \label{appendix:enum_Fn} 
\item the functions \( \omega \mapsto \nabla F_{n,k} \in C_t^0 \Bkprime \),   \( \omega \mapsto \partial_t F_{n,k} \in C_t^0 \Bkprime \), and  \( \omega \mapsto  F_{n,k} \in C_t^0 \Bkprime \)  are strongly \( \mathscr{F}_{n-1}\)- measurable, \label{appendix:enum_Fnk}
\item the function \( \omega \mapsto w_n \in \YN \) is strongly \( \mathscr{F}_n \)-measurable. \label{appendix:enum_wn}
\end{enumerate}
Furthermore, let \( u \) be the solution from Theorem \ref{main_thm}. Then, the maps \( \omega \mapsto u\in C_t^0 H_x^s \medcap L_t^2 W_x^{\sigma,\infty} \) and \( \omega \mapsto \partial_t u \in C_t^0 H_x^{s-1} \) are strongly \( \mathbb{P}\)-measurable. 
\end{prop}

Before we prove the proposition, we need the following lemma which proves the measurability of the cutoff. 

\begin{lem}\label{appendix:lem_cutoff}
If \( \omega \in \Omega \mapsto v^\omega \in \XN([0,T]) \) is strongly \( \mathscr{F}_n \)-measurable, then the map 
\begin{equation*}
(\omega,\tau) \in \Omega \times [0,T] \rightarrow \| v^\omega \|_{\XN([0,\tau])} \in \mathbb{R}_{\geq 0} 
\end{equation*}
is measurable with respect to the product sigma-algebra \( \mathscr{F}_n \bigotimes \mathscr{B}([0,T]) \). Here, \( \mathscr{B}([0,T]) \) denotes the Borel sigma-algebra. \\
An analogous statement also holds for \( \XlN([0,T]), ~ \SN([0,T]), ~ \text{and} ~ \SlN([0,T]) \). 
\end{lem}
\begin{proof}
Since \( v^\omega \) is strongly \( \mathscr{F}_n \)-measurable, it suffices to prove the statement for simple functions. Thus, we may assume that there exists pairwise disjoint measurable sets \( F_i \in \mathscr{F}_n \) and (deterministic) functions \( v_i \in \XN([0,T]) \), \( i=1,\hdots, k \), such that 
\begin{equation*}
v^\omega  = \sum_{i=1}^k 1_{F_i}(\omega) v_i ~. 
\end{equation*} 
It follows that \begin{equation*}
\| v^\omega \|_{\XN([0,\tau])} = \sum_{i=1}^k 1_{F_i}(\omega) \| v_i \|_{\XN([0,\tau])}~. 
\end{equation*}
Thus, Lemma \ref{appendix:lem_cutoff}  follows from the continuity statement \eqref{prelim:eq_continuity}.
\end{proof}

\begin{proof}[Proof of Proposition \ref{appendix:prop_measurability}]
We prove the proposition by induction on \( n\). Since the base case \( n=0 \) and induction step follow from the same argument, we may assume directly that \ref{appendix:enum_Fn}-\ref{appendix:enum_wn} hold for all \( {m=0,\hdots, n-1} \). \\
Due to \ref{appendix:enum_Fn}, \ref{appendix:enum_wn}, and Lemma \ref{appendix:lem_properties}.\ref{appendix:enum_cont}, we see that \( \omega \mapsto  u_{n-1} \in L_t^2 W_x^{\sigma,\infty} \) is strongly \(\mathscr{F}_{n-1} \)-measurable. Similarly, using \ref{appendix:enum_Fn}, \ref{appendix:enum_wn}, and Lemma \ref{appendix:lem_cutoff}, we obtain the measurability of the cutoff \( \theta_{F,w;\leq n-1} \). Since the proof of  Proposition \ref{lin:prop_freq_envelope} leads to a contraction mapping argument, we see that the solution \( F_k \) of \eqref{lin:eq_Fk} depends continuously on \( \phi \in L_t^2 W_x^{\sigma,\infty} \). Therefore, we obtain \ref{appendix:enum_Fnk} from Lemma \ref{appendix:lem_properties}.\ref{appendix:enum_cont}. For any sufficiently large \( D^{\prime\prime} >0 \), we have the continuous embeddings \( C_t^0 \Bkprime \hookrightarrow \XNprime \) and \( C_t^0 \Bkprime \hookrightarrow \SNprime\), where the norm of the embedding may depend on \( N \). Since \begin{equation*}
F_n^\omega = \sum_{N/2 \leq \| k\|_2 < N} g_k(\omega) F_{n,k}~,
\end{equation*}
this proves \ref{appendix:enum_Fn}.  Since the proof of Proposition \ref{nl:prop_wn} consists of a contraction mapping argument, \( w_n \in \YN \) depends continuously on \( F_m^\omega\), \( w_m \), where \( m=0,\hdots, n-1\), and \( F_n^\omega\), all in their respective norms. Thus, \ref{appendix:enum_wn} follows from \ref{appendix:enum_Fn} with \( m=0,\hdots,n \), \ref{appendix:enum_wn} with \( m=0,\hdots,n-1\),  Lemma \ref{appendix:lem_properties}, and Lemma \ref{appendix:lem_cutoff}. \\
Finally, the strong \( \mathbb{P} \)-measurability of \( u \) follows from the convergence of the iterates, see \eqref{fin:eq_convergence}.
\end{proof}

\bibliographystyle{hplain}
\bibliography{Library_Wiki}

\begin{thebibliography}{10}

\bibitem{BOP15}
\'Arp\'ad B\'enyi, Tadahiro Oh, and Oana Pocovnicu.
\newblock On the probabilistic {C}auchy theory of the cubic nonlinear
  {S}chr\"odinger equation on {$\Bbb{R}^d$}, {$d\geq3$}.
\newblock {\em Trans. Amer. Math. Soc. Ser. B}, 2:1--50, 2015.

\bibitem{BOP17}
\'Arp\'ad B\'enyi, Tadahiro Oh, and Oana Pocovnicu.
\newblock {Higher order expansions for the probabilistic local Cauchy theory of
  the cubic nonlinear Schrödinger equation on $\mathbb{R}^3$}, September 2017,
  arXiv:1709.01910.

\bibitem{BOP18}
\'Arp\'ad B\'enyi, Tadahiro Oh, and Oana Pocovnicu.
\newblock {On the probabilistic Cauchy theory for nonlinear dispersive PDEs},
  May 2018, arXiv:1805.08411.

\bibitem{Bourgain93}
Jean Bourgain.
\newblock Fourier transform restriction phenomena for certain lattice subsets
  and applications to nonlinear evolution equations. {I}. {S}chr\"odinger
  equations.
\newblock {\em Geom. Funct. Anal.}, 3(2):107--156, 1993.

\bibitem{Bourgain94}
Jean Bourgain.
\newblock Periodic nonlinear {S}chr\"odinger equation and invariant measures.
\newblock {\em Comm. Math. Phys.}, 166(1):1--26, 1994.

\bibitem{Bourgain96}
Jean Bourgain.
\newblock Invariant measures for the {$2$}{D}-defocusing nonlinear
  {S}chr\"odinger equation.
\newblock {\em Comm. Math. Phys.}, 176(2):421--445, 1996.

\bibitem{Bourgain97}
Jean Bourgain.
\newblock Invariant measures for the {G}ross-{P}iatevskii equation.
\newblock {\em J. Math. Pures Appl. (9)}, 76(8):649--702, 1997.

\bibitem{BB14}
Jean Bourgain and Aynur Bulut.
\newblock Invariant {G}ibbs measure evolution for the radial nonlinear wave
  equation on the 3d ball.
\newblock {\em J. Funct. Anal.}, 266(4):2319--2340, 2014.

\bibitem{Brereton16}
Justin~T. {Brereton}.
\newblock {Almost sure local well-posedness for the supercritical quintic NLS},
  December 2016, arXiv:1612.05366.

\bibitem{Bringmann18}
Bjoern {Bringmann}.
\newblock {Almost sure scattering for the radial energy critical nonlinear wave
  equation in three dimensions}, April 2018, arXiv:1804.09268.

\bibitem{BT08I}
Nicolas Burq and Nikolay Tzvetkov.
\newblock Random data {C}auchy theory for supercritical wave equations. {I}.
  {L}ocal theory.
\newblock {\em Invent. Math.}, 173(3):449--475, 2008.

\bibitem{BT08II}
Nicolas Burq and Nikolay Tzvetkov.
\newblock Random data {C}auchy theory for supercritical wave equations. {II}.
  {A} global existence result.
\newblock {\em Invent. Math.}, 173(3):477--496, 2008.

\bibitem{BT14}
Nicolas Burq and Nikolay Tzvetkov.
\newblock Probabilistic well-posedness for the cubic wave equation.
\newblock {\em J. Eur. Math. Soc. (JEMS)}, 16(1):1--30, 2014.

\bibitem{CCMNS18}
Sagun {Chanillo}, Magdalena {Czubak}, Dana {Mendelson}, Andrea {Nahmod}, and
  Gigliola {Staffilani}.
\newblock {Almost sure boundedness of iterates for derivative nonlinear wave
  equations}, October 2017, arXiv:1710.09346.

\bibitem{PD02}
Giuseppe Da~Prato and Arnaud Debussche.
\newblock Two-dimensional {N}avier-{S}tokes equations driven by a space-time
  white noise.
\newblock {\em J. Funct. Anal.}, 196(1):180--210, 2002.

\bibitem{DFS12}
Piero D'Ancona, Damiano Foschi, and Sigmund Selberg.
\newblock Atlas of products for wave-{S}obolev spaces on {$\Bbb{R}^{1+3}$}.
\newblock {\em Trans. Amer. Math. Soc.}, 364(1):31--63, 2012.

\bibitem{BD99}
Anne de~Bouard and Arnaud Debussche.
\newblock A stochastic nonlinear {S}chr\"odinger equation with multiplicative
  noise.
\newblock {\em Comm. Math. Phys.}, 205(1):161--181, 1999.

\bibitem{S11}
Anne-Sophie de~Suzzoni.
\newblock Invariant measure for the cubic wave equation on the unit ball of
  {$\Bbb R^3$}.
\newblock {\em Dyn. Partial Differ. Equ.}, 8(2):127--147, 2011.

\bibitem{Deng15}
Yu~Deng.
\newblock Invariance of the {G}ibbs measure for the {B}enjamin-{O}no equation.
\newblock {\em J. Eur. Math. Soc. (JEMS)}, 17(5):1107--1198, 2015.

\bibitem{DLM17}
Benjamin {Dodson}, Jonas {Lührmann}, and Dana {Mendelson}.
\newblock {Almost sure scattering for the 4D energy-critical defocusing
  nonlinear wave equation with radial data}, March 2017, arXiv:1703.09655.

\bibitem{DLM18}
Benjamin {Dodson}, Jonas {Lührmann}, and Dana {Mendelson}.
\newblock {Almost sure local well-posedness and scattering for the 4D cubic
  nonlinear Schrödinger equation}, February 2018, arXiv:1802.03795.

\bibitem{FK00}
Damiano Foschi and Sergiu Klainerman.
\newblock Bilinear space-time estimates for homogeneous wave equations.
\newblock {\em Ann. Sci. \'Ecole Norm. Sup. (4)}, 33(2):211--274, 2000.

\bibitem{GN14}
Viktor Grigoryan and Andrea~R. Nahmod.
\newblock Almost critical well-posedness for nonlinear wave equations with
  {$Q_{\mu\nu}$} null forms in 2{D}.
\newblock {\em Math. Res. Lett.}, 21(2):313--332, 2014.

\bibitem{GT13}
Viktor {Grigoryan} and Allison {Tanguay}.
\newblock {Improved well-posedness for the quadratic derivative nonlinear wave
  equation in 2D}, August 2013, arXiv:1308.1719.

\bibitem{Grunrock11}
Axel Gr\"{u}nrock.
\newblock On the wave equation with quadratic nonlinearities in three space
  dimensions.
\newblock {\em J. Hyperbolic Differ. Equ.}, 8(1):1--8, 2011.

\bibitem{KT98}
Markus Keel and Terence Tao.
\newblock Endpoint {S}trichartz estimates.
\newblock {\em Amer. J. Math.}, 120(5):955--980, 1998.

\bibitem{KM93}
Sergiu Klainerman and Matei Machedon.
\newblock Space-time estimates for null forms and the local existence theorem.
\newblock {\em Comm. Pure Appl. Math.}, 46(9):1221--1268, 1993.

\bibitem{KRT02}
Sergiu Klainerman, Igor Rodnianski, and Terence Tao.
\newblock A physical space approach to wave equation bilinear estimates.
\newblock {\em J. Anal. Math.}, 87:299--336, 2002.

\bibitem{KT99}
Sergiu Klainerman and Daniel Tataru.
\newblock On the optimal local regularity for {Y}ang-{M}ills equations in
  {${\bf R}^{4+1}$}.
\newblock {\em J. Amer. Math. Soc.}, 12(1):93--116, 1999.

\bibitem{Lindblad93}
Hans Lindblad.
\newblock A sharp counterexample to the local existence of low-regularity
  solutions to nonlinear wave equations.
\newblock {\em Duke Math. J.}, 72(2):503--539, 1993.

\bibitem{Lindblad96}
Hans Lindblad.
\newblock Counterexamples to local existence for semi-linear wave equations.
\newblock {\em Amer. J. Math.}, 118(1):1--16, 1996.

\bibitem{LM13}
Jonas L\"uhrmann and Dana Mendelson.
\newblock Random data {C}auchy theory for nonlinear wave equations of
  power-type on {$\Bbb {R}^3$}.
\newblock {\em Comm. Partial Differential Equations}, 39(12):2262--2283, 2014.

\bibitem{LM16}
Jonas L\"uhrmann and Dana Mendelson.
\newblock On the almost sure global well-posedness of energy sub-critical
  nonlinear wave equations on {$\Bbb R^3$}.
\newblock {\em New York J. Math.}, 22:209--227, 2016.

\bibitem{OP16}
Tadahiro Oh and Oana Pocovnicu.
\newblock Probabilistic global well-posedness of the energy-critical defocusing
  quintic nonlinear wave equation on {$\Bbb{R}^3$}.
\newblock {\em J. Math. Pures Appl. (9)}, 105(3):342--366, 2016.

\bibitem{Pocovnicu17}
Oana Pocovnicu.
\newblock Almost sure global well-posedness for the energy-critical defocusing
  nonlinear wave equation on {$\Bbb{R}^d$}, {$d=4$} and {$5$}.
\newblock {\em J. Eur. Math. Soc. (JEMS)}, 19(8):2521--2575, 2017.

\bibitem{PS93}
Gustavo Ponce and Thomas~C. Sideris.
\newblock Local regularity of nonlinear wave equations in three space
  dimensions.
\newblock {\em Comm. Partial Differential Equations}, 18(1-2):169--177, 1993.

\bibitem{SS98}
Jalal Shatah and Michael Struwe.
\newblock {\em Geometric wave equations}, volume~2 of {\em Courant Lecture
  Notes in Mathematics}.
\newblock New York University, Courant Institute of Mathematical Sciences, New
  York; American Mathematical Society, Providence, RI, 1998.

\bibitem{Tao01}
Terence Tao.
\newblock Global regularity of wave maps. {I}. {S}mall critical {S}obolev norm
  in high dimension.
\newblock {\em Internat. Math. Res. Notices}, (6):299--328, 2001.

\bibitem{Tao04}
Terence Tao.
\newblock Global well-posedness of the {B}enjamin-{O}no equation in {$H^1({\bf
  R})$}.
\newblock {\em J. Hyperbolic Differ. Equ.}, 1(1):27--49, 2004.

\bibitem{Tao06}
Terence Tao.
\newblock {\em Nonlinear dispersive equations}, volume 106 of {\em CBMS
  Regional Conference Series in Mathematics}.
\newblock Published for the Conference Board of the Mathematical Sciences,
  Washington, DC; by the American Mathematical Society, Providence, RI, 2006.
\newblock Local and global analysis.

\bibitem{Tataru99}
Daniel Tataru.
\newblock On the equation {$\square u=|\nabla u|^2$} in {$5+1$} dimensions.
\newblock {\em Math. Res. Lett.}, 6(5-6):469--485, 1999.

\bibitem{TG05}
Daniel Tataru and Dan-Andrei Geba.
\newblock Dispersive estimates for wave equations.
\newblock {\em Comm. Partial Differential Equations}, 30(4-6):849--880, 2005.

\bibitem{Zhou97}
Yi~Zhou.
\newblock Local existence with minimal regularity for nonlinear wave equations.
\newblock {\em Amer. J. Math.}, 119(3):671--703, 1997.

\bibitem{Zhou03}
Yi~Zhou.
\newblock On the equation {$\square\phi=|\nabla\phi|^2$} in four space
  dimensions.
\newblock {\em Chinese Ann. Math. Ser. B}, 24(3):293--302, 2003.

\end{thebibliography}

\Addresses

\end{document}